\documentclass[12pt,a4paper]{amsart}

\usepackage{amsthm,amssymb,verbatim,amsmath,amsfonts, amscd, color} 
\usepackage{tikz-cd}

\usepackage[left=25mm,right=25mm,top=30mm,bottom=30mm]{geometry}
\usepackage{enumitem}
\setlist{leftmargin=10mm, topsep=1mm, itemsep=1pt}

\usepackage[sc,outermarks]{titlesec}
\titleformat{\section}{\normalsize\sc\center}{\thesection}{0mm}{ }
\titlespacing{\section}{0mm}{3em}{1em}

\usepackage{hyperref} %
\hypersetup{
    colorlinks=true,       
    linkcolor=red,          
    citecolor=black,        
    urlcolor=blue        
}

\theoremstyle{plain}\newtheorem{Theorem}{Theorem}[section]
\theoremstyle{plain}
\theoremstyle{plain}\newtheorem{Corollary}[Theorem]{Corollary}
\theoremstyle{plain}\newtheorem{Lemma}[Theorem]{Lemma}
\theoremstyle{plain}\newtheorem{Proposition}[Theorem]{Proposition}
\theoremstyle{definition}
\theoremstyle{definition}
\theoremstyle{definition}
\theoremstyle{definition}\newtheorem{Remark}[Theorem]{Remark}
\theoremstyle{definition}
\theoremstyle{definition}\newtheorem{Notation}[Theorem]{Notation}
\theoremstyle{definition}\newtheorem{Hypothesis}[Theorem]{Hypothesis}
\theoremstyle{definition}
\theoremstyle{definition}
\theoremstyle{definition}
\theoremstyle{definition}
\theoremstyle{definition}
\theoremstyle{definition}\newtheorem{Notation/Definition}[Theorem]{Notation/Definition}
\theoremstyle{definition}
\theoremstyle{plain}

\def\Aut{\mathrm{Aut}} 
\def\Syl{\mathrm{Syl}}
\def\im{\mathrm{im}}
\def\rad{\mathrm{rad}}

\def\dim{\mathrm{dim}}
 \def\End{\mathrm{End}} 
\def\Ext{\mathrm{Ext}}
  \def\Hom{\mathrm{Hom}}

\def\ker{\mathrm{ker}}

 \def\Irr{\mathrm{Irr}}
 \def\IBr{\mathrm{IBr}} 

\def\mod{\mathrm{mod}} \def\rad{\mathrm{rad}}

\def\soc{\mathrm{soc}}

\newcommand{\hd}{\operatorname{hd}}

\newcommand{\PSL}{\operatorname{PSL}}
\newcommand{\SL}{\operatorname{SL}}
\newcommand{\PGL}{\operatorname{PGL}}
\newcommand{\GL}{\operatorname{GL}}
\newcommand{\PSU}{\operatorname{PSU}}
\newcommand{\SU}{\operatorname{SU}} 
\newcommand{\PGU}{\operatorname{PGU}}
\newcommand{\GU}{\operatorname{GU}}

\newcommand{\Sc}{{\text{Sc}}}

\newcommand{\Out}{{\text{Out}}} 

\begin{document}

\title[Principal $2$-blocks with wreathed defect groups up to splendid Morita equivalence]
{Principal $2$-blocks with wreathed defect groups up to splendid Morita equivalence}\date{\today}
\author{Shigeo Koshitani, Caroline Lassueur and Benjamin Sambale} 

\dedicatory{To the memory of Professor Atumi Watanabe}

\address{S. Koshitani, Department of Mathematics and Informatics,
Graduate School of Science, Chiba
University, 1-33 Yayoi-cho, Inage-ku,Chiba 263-8522, Japan}
\email{koshitan@math.s.chiba-u.ac.jp} 
\address{C. Lassueur, 
Institut f{\"u}r Algebra, Zahlentheorie und Diskrete Mathematik, Leibniz
Universit{\"a}t Hannover, Welfengarten 1, D-30167 Hannover
Germany}
\email{lassueur@mathematik.uni-kl.de} 
\address{B. Sambale, Institut f{\"u}r Algebra, Zahlentheorie und Diskrete Mathematik, Leibniz
Universit{\"a}t Hannover, Welfengarten 1, D-30167 Hannover, Germany} 
\email{sambale@math.uni-hannover.de} 
\thanks{The first author
was partially supported by the Japan Society for Promotion of Science (JSPS),
Grant-in-Aid for Scientific Research (C)19K03416, 2019--2021, 
and also by the Research Institute for Mathematical Sciences (RIMS),
an International Joint Usage/Research Center located in Kyoto University.
The second author
was supported by the DFG SFB/TRR 195.} 
\keywords{wreathed $2$-group, Morita equivalence, splendid Morita equivalence, Puig's Finiteness 
Conjecture, principal block,  trivial source module, $p$-permutation module, Scott module, 
Brauer indecomposability, decomposition matrix}
\subjclass[2010]{20C05, 20C20, 20C15, 20C33,16D90} 

 
\begin{abstract} 
We classify principal $2$-blocks of finite groups $G$ with Sylow $2$-subgroups isomorphic 
to a wreathed $2$-group $C_{2^n}\wr C_2$ with $n\geq 2$ up to Morita equivalence  and 
up to splendid Morita equivalence. As a consequence,  we obtain 
that Puig's Finiteness Conjecture holds for such blocks.
Furthermore, we obtain a classification of such groups modulo $O_{2'}(G)$, 
which is a purely group theoretical result and of independent interest. 
Methods previously applied to blocks of tame representation type are used. They are, however, further
developed in order to deal with blocks of wild representation type.
 \end{abstract}

\maketitle

\pagestyle{myheadings} 
\markboth{S. Koshitani, C. Lassueur and B. Sambale}
{Principal $2$-blocks with wreathed defect groups up to splendid Morita equivalence}


\section{Introduction}
Let $k$ be an algebraically closed field of positive characteristic~$p$. 
A splendid Morita equivalence between two block algebras $B_1$ and $B_2$ 
of finite groups $G_1$ and $G_2$ of order divisible by $p$ is 
a Morita equivalence which is induced by a $(B_1,B_2)$-bimodule 
(and its dual) which is a $p$-permutation module when regarded 
as a one-sided $k(G_1\times G_2)$-module. 
Such equivalences play an important role in the modular representation 
theory of finite groups as they  preserve many important invariants 
such as the defect groups or the generalised decomposition numbers, 
and encode the structure of the source algebras. In this respect,  
Puig's finiteness conjecture
(see \cite[6.2]{Bro94} or \cite[Conjecture 6.4.2]{Lin18})
extends Donovan's conjecture to 
include the structure of the source algebra of $p$-blocks 
and  postulates that given  a  finite $p$-group~$D$,  there  are  only finitely  
many  interior $D$-algebras, up  to  isomorphism,   
which  are  source algebras of $p$-blocks of 
finite groups with defect group~$D$. This is equivalent 
to postulating that there are only finitely many splendid Morita equivalence 
classes of $p$-blocks of finite groups with defect group~$D$.
\par
In a series of previous articles \cite{KL20a,KL20b,KLS20} 
the authors classified principal block algebras of tame representation type 
up to splendid Morita equivalence, that is, in the case in which~$p=2$ and 
the Sylow $2$-subgroups of the groups considered are either dihedral, semi-dihedral, or 
generalised quaternion $2$-groups. As a corollary, the validity of 
Puig's conjecture is verified for this class of $2$-blocks. 
(The tame domestic case was settled in \cite{CEKL11}.) 
The aim of the present article is to give a first try at applying similar methods 
in {\it wild} representation type under good hypotheses: 
we investigate here groups with Sylow $2$-subgroups isomorphic to 
a wreathed $2$-group $C_{2^n}\wr C_2$ with $n\geq 2$.  
We choose this defect group for its many similarities with the tame cases. 
In this respect, from the group theory point of view,  we strongly rely on the 
facts that the wreathed $2$-groups $C_{2^n}\wr C_2$ have $2$-rank $2$ 
and an automorphism group which is a $2$-group, whereas from the modular 
representation theory point of view we rely on the Brauer indecomposability 
of Scott modules with wreathed vertices proved by the first author and Tuvay in~\cite{KT20}.  
\par
In order to state our main results, we first need to introduce some notation.  
Given a finite group $G$ and $H\leq G$,  we set $\Delta H:=\{(h,h)\in G\times G\mid h\in H\}$ 
and we recall that the \emph{Scott module} of $kG$ with respect to $H$, 
denoted  by $\Sc(G,H)$, is, up to isomorphism, the unique indecomposable direct summand of 
the trivial $kH$-module induced from $H$ to $G$ with the property that the trivial 
$kG$-module is a constituent of its head (or equivalently of its socle). Furthermore, 
given an integer $t\geq 0$ and a power $q=r^{f}$ of a prime number with $f\geq 1$ an integer, we let 
$$ \SL^t_2(q):
=\{A\in{\mathrm{GL}}_2(q)\,|\,\det(A)^{2^t}=1 \} \
 \text{ and } \ {\mathrm{SU}}^t_2(q):
=\{A\in{\mathrm{GU}}_2(q)\,|\, \det(A)^{2^t}=1 \}\,.$$
Now, in order to apply the previously developed methods, our 
first main result provides a classification 
of the finite groups $G$ with a wreathed Sylow $2$-subgroup $C_{2^n}\wr C_2$ ($n\geq 2$) 
modulo $O_{2'}(G)$, which is of independent interest.

\begin{Theorem}\label{thm:classification_wr} 
Let $G$ be a finite group with a Sylow $2$-subgroup
isomorphic to a wreathed $2$-group $C_{2^n}\wr C_2$ for an integer $n\geq 2$ 
such that $O_{2'}(G) = 1$. 
Then one of the following \smallskip holds: 
\begin{enumerate}[leftmargin=17mm]
\item[{\rm\sf(WR1)}] $G\cong
C_{2^n}\wr C_2$\,,
\item[{\rm\sf (WR2)}] $G
\cong (C_{2^n}\times C_{2^n})\rtimes \mathfrak S_3$\,,
\item[{\rm\sf(WR3)}]
$G\cong{\mathrm{SL}}{\color{black}^n_2}(q) \rtimes C_d$
where $q=r^{f}$ is a power of a prime number  $r$ with $f$ a positive integer, $(q-1)_2=2^n$ and $d$ is an odd integer such that  $d\mid f$\,;
\item[{\rm\sf(WR4)}]
$G\cong{\mathrm{SU}}{\color{black}^n_2}(q) \rtimes C_d$ where $q=r^{f}$ is a power of a prime number  $r$ with $f$ a positive integer,  $(q+1)_2=2^n$ and $d$ is an odd integer such that  $d\mid f$\,;
\item[{\rm\sf(WR5)}] $G\cong\PSL_3 (q).H$ where $q=r^{f}$ is a power of a prime number  $r$ with $f$ a positive integer,  $(q-1)_2=2^n$, 
$H \leq C_{(q-1, 3)}\times C_d$ and $d$ is an odd integer such that $d\mid f$\,; or
\item[{\rm\sf(WR6)}]
$G\cong\PSU_3(q).H$ where where $q=r^{f}$ is a power of a prime number  $r$ with $f$ a positive integer,  $(q+1)_2=2^n$, $H \leq
C_{(q+1, 3)}\times C_d$  and $d$ is an odd integer such that $d\mid f$\,.
\end{enumerate}
\end{Theorem}

This theorem, which we prove in  Section~\ref{sec:wreathedSylow},
 is a byproduct of  Alperin--Brauer--Gorenstein's work \cite{ABG70} on finite groups with 
quasi-dihedral and wreathed Sylow $2$-subgroups.
\par 
Our second main result is then a classification of principal blocks with  defect groups isomorphic to 
a wreathed $2$-group $C_{2^n}\wr C_2$ with $n\geq 2$.

\begin{Theorem}\label{thm:MainResult}%
Let $k$ be an algebraically closed field of characteristic~$2$ and 
let $G$ be a finite group with a Sylow $2$-subgroup $P$ 
isomorphic to a wreathed $2$-group $C_{2^n}\wr C_2$ for a fixed integer $n\geq 2$. 
Then, the following assertions hold.
\begin{enumerate}[label={\rm(\alph*)}]
\item 
The principal $2$-block $B_0(kG)$ of $G$ is splendidly Morita equivalent to the
principal $2$-block $B_0(kG')$ of a finite group $G'$ belonging to precisely 
one of the following families of  finite groups:
\begin{itemize}[leftmargin=18mm]
\item[{\rm\sf (W1(n))}] $C_{2^n}\wr C_2$\,; 
\item[{\rm\sf (W2(n))}] $
(C_{2^n}\times C_{2^n})\rtimes \mathfrak S_3$\,;  
\item[{\rm\sf(W3(n))}]
${\mathrm{SL}}{\color{black}^n_2}(q)$ where  $q$ is a power of a prime number
such that  $(q-1)_2=2^n$;
\item[{\rm\sf (W4(n))}] ${\mathrm{SU}}{\color{black}^n_2}(q)$  where  
$q$ is a power of a prime number such that $(q+1)_2=2^n$; 
\item[{\rm\sf (W5(n))}] 
$\PSL_3(q)$   where  $q$ is a power of a prime number such that $(q-1)_2=2^n$; or
\item[{\rm\sf (W6(n))}] $\PSU_3(q)$  where  $q$ is a power of a prime number
such that $(q+1)_2=2^n$. 
\end{itemize}
Moreover, in all cases, the splendid Morita equivalence is induced by the
Scott module ${{\mathrm{Sc}}(G\times G', \Delta P)}$, where $P$ is also seen as a Sylow $2$-subgroup of~$G'$.
\item 
In {\rm(a)}, more accurately, if $G_1$ and $G_2$ are two finite groups belonging 
to the same infinite family of finite groups {\sf(Wj(n))} with ${\sf j}\in \{{\sf 3},{\sf 4},{\sf 5},{\sf 6}\}$,
then ${{\mathrm{Sc}}(G_{1}\times G_2, \Delta P)}$ induces a splendid Morita equivalence  between 
$B_0(kG_1)$ and $B_0(kG_2)$.
\end{enumerate} 
\end{Theorem}

We emphasize that in the case of principal blocks of tame representation type, treated 
in~\cite{KL20a,KL20b,KLS20}, a classification of these blocks up to Morita equivalence 
was known by Erdmann's work on tame algebras~\cite{Erd90}. 
A major difference in the case of wreathed Sylow $2$-subgroups lies in the fact that
a classification of these blocks up to Morita equivalence was, to our knowledge, not known.
However, it follows from our methods, that the classification up to splendid Morita equivalence 
which we have obtained coincides with the classification up to Morita equivalence. 

\begin{Theorem}\label{thm:MoritaClasses}%
Let $k$ be an algebraically closed field of characteristic~$2$ and 
let $G$ be a finite group with a Sylow $2$-subgroup  
isomorphic to a wreathed $2$-group $C_{2^n}\wr C_2$ for a fixed integer $n\geq 2$. 
Then $B_{0}(kG)$ is Morita equivalent to the principal block of precisely one of the families of groups 
{\rm\sf (W1(n))}, {\rm\sf (W2(n))}, {\rm\sf (W3(n))}, {\rm\sf (W4(n))}, {\rm\sf (W5(n))}, or {\rm\sf (W6(n))}
as in  Theorem~\ref{thm:MainResult}(a).
\end{Theorem}

As an immediate consequence of Theorem~\ref{thm:MainResult} we also obtain that 
Puig's Finiteness Conjecture holds if we restrict our attention to principal blocks with a 
defect group isomorphic to a wreathed $2$-group $C_{2^n}\wr C_2$. 

\begin{Corollary}\label{cor:MainResult} 
For each integer $n\geq 2$ there are only finitely many 
splendid Morita equivalence classes of  
principal $2$-blocks with defect groups isomorphic to 
a wreathed $2$-group $C_{2^n}\wr C_2$. 
\end{Corollary}

This paper is organised as follows. In Section~\ref{sec:nota} 
the notation is introduced. In Section~\ref{sec:wreathedSylow} 
we state and prove the classification of finite groups $G$ with a wreathed Sylow $2$-subgroup 
and $O_{2'}(G)=1$. In Section~\ref{sec:pre} 
we recall, state and prove preliminary results on 
splendid Morita equivalences and on module theory over finite-dimensional algebras.
In Sections~\ref{sec:W3W4}, \ref{sec:W5} and \ref{sec:W6} we prove part (b) 
of Theorem~\ref{thm:MainResult}.  Section~\ref{sec:proofMainb} contains the proof of Theorem~\ref{thm:MainResult} 
and Theorem~\ref{thm:MoritaClasses}. Finally, Appendix~\ref{App:A} fixes a gap in the proof of~\cite[Proposition~3.3(b)]{KL20b}.

\section{Notation}\label{sec:nota}
Throughout this paper, unless otherwise stated, we adopt the following notation and 
conventions.  We let $k$ be an algebraically closed field of characteristic~{\color{black}$p>0$}. 
All groups considered are finite,  all $k$-algebras are finite-dimensional 
and all modules over finite-dimensional algebras considered are finitely generated right modules.
The symbols $G$, $G'$, $G_{1}$, $G_{2}$,  $\mathcal G_{1}$ and  $\mathcal G_{2}$ 
always denote finite groups of order divisible by $p$.
\par  
Furthermore, we denote by ${\mathrm{Syl}}_p(G)$ the set of all Sylow $p$-subgroups of $G$, and 
for $P\in\Syl_p(G)$, we let $\mathcal F_P(G)$ be the fusion system of $G$ on $P$.
If $H\leq G$, we let $\Delta H:=\{(h, h)\in G\times G\,|\, h\in H\}$ denote the diagonal 
embedding of $H$ in $G\times G$. 
Given an integer $m\geq 2$, we let $D_{2^m}$ denote the dihedral group of order $2^m$,  
$C_m$ denote the cyclic group of order $m$, and  $C_{2^m}\wr C_2$ denote the 
wreathed product of $C_{2^m}$ by $C_2$.
Given an integer $t\geq 0$ and a positive prime power~$q$, we let 
$${\mathrm{SL}}^t_2(q):=\{A\in{\mathrm{GL}}_2(q)\,|\,
\det(A)^{2^t}=1 \} \ \text{ and } \ {\mathrm{SU}}^t_2(q)
:=\{A\in{\mathrm{GU}}_2(q)\,|\, \det(A)^{2^t}=1 \}\,,$$
as already defined in the introduction.
\par
Given  a finite-dimensional $k$-algebra $A$, we denote by $\rad(A)$ the Jacobson radical  of~$A$
and by $1_A$ the unit element of $A$, respectively.
Furthermore, if $X$ is an $A$-module and  $m\geq 0$ is an integer, then we denote by $\soc^m(X):={\{x\in X \mid x\cdot\rad(A)^m=0\}}$  
the $m$-th socle of $X$, where $\soc(X):=\soc^1(X)$ is the socle of $X$, and for $1\leq i\leq \ell$, 
where $\ell$ is the Loewy (or radical) length of $X$, we set 
$$S_i(X):=\soc^i(X)/\soc^{i-1}(X)\quad\text{ and }\quad L_i(X):= X\,{\mathrm{rad}}(A)^{i-1}/X\,{\mathrm{rad}}(A)^i$$
and we write $\text{hd}(X)$ for the head of $X$. 
We then talk about the \emph{radical (Loewy) series} and about the \emph{socle series} of $X$ as defined in 
\cite[Chap. I \S8]{Lan83}. 
We describe a uniserial $A$-module $X$ with simple composition factors $L_i(X)\cong S_i$ 
for simple $A$-modules $S_1, \cdots, S_{\ell}$ via the diagram
$$X= {\boxed{ \begin{matrix}S_1\\ \vdots\\S_{\ell}\end{matrix}} \,.}$$
We denote by $P(X)$ the projective cover of an $A$-module $X$ and
by $\Omega(X)$ the kernel of the canonical morphism $P(X)\twoheadrightarrow X$. 
Dually, we let $\Omega^{-1}(X):=I(X)/X$ where $I(X)$ is an 
injective envelope 
of $X$,
and we denote by $X^*$ the $k$-dual of $X$ (which is a left $A$-module).
Given a simple $A$-module $S$, we denote by $c_X(S)$ 
the multiplicity of $S$ as a composition factor of $X$ and if $S_1, \cdots, S_n$ 
are all the pairwise non-isomorphic composition factors of $X$ 
with multiplicities $m_{1},\ldots,m_{n}$, respectively, then we write 
$X=m_1\times S_1+\cdots +m_n\times S_n$ (as composition factors).
If $Y$ is another $A$-module, then $Y\mid X$ (resp. $Y\nmid X$) means that $Y$ is 
isomorphic (resp. not isomorphic) to a direct summand of $X$, $({\mathrm{proj}})$ 
denotes a projective $A$-module (which we do not need to specify). 
\par
We write $B_0(kG)$ for the principal block of the group algebra $kG$. 
Given a  block $B$ of $kG$, we write $1_B$ for the block idempotent of $B$ and $C_B$ for 
the Cartan matrix of $B$.
We denote by $\Irr(B)$ and $\IBr(B)$, respectively, the sets of all
irreducible ordinary and Brauer characters of $G$ belonging to $B$.
If $D\leq G$ is a defect group of the block $B$, then the integer $d$ such that $|D|=p^{d}$ is called the defect of $B$.
Assuming $|G| = p^{a}m$ with $p\nmid m$, if $\chi\in\Irr(G)$ lies in a block of defect $d$, then 
the height of $\chi$, denoted by $\text{ht}(\chi)$, is defined to be the exact power of $p$ dividing the integer
$\chi(1)/ p^{a-d}$. 
We write $k(B):=|\Irr(B)|$ and $\ell(B):=|\IBr(B)|$ and 
$k_{i}(B):=|\{\chi\in\Irr(B)\mid \text{ht}(\chi)=i\}|$ where $\text{ht}(\chi)$ 
is the height of~$\chi$. 
\par
We denote by 
$k_G$ the trivial $kG$-module. Given a $kG$-module $M$ and 
a $p$-subgroup $Q\leq G$ we denote by 
$M(Q)$ the Brauer construction of $M$ with respect to $Q$. (See e.g. \cite[p. 219]{The95}.)
When $H\leq G$, $N$ is a $kH$-module and $M$ is a $kG$-module,
we write $N{\uparrow}^G$ and $M{\downarrow}_H$ respectively for
the induction of $N$ to~$G$ and the restriction of $M$ to $H$.
For a subgroup $H\leq G$ 
we denote by ${\mathrm{Sc}}(G,H)$ the Scott module of $kG$ with 
respect to~$H$, which by 
definition is the unique indecomposable direct summand of 
${k_H}{\uparrow}^G$ (up to isomorphism) that has the trivial module $k_G$ as a constituent of its
head (or equivalently of its socle). This is a $p$-permutation module (see
\cite[Chapter 4, \S 8.4]{NT89}).
\par
If  $B_1$ and $B_2$ are two  finite-dimensional $k$-algebras and 
$M$ is a $(B_1,B_2)$-bimodule, we also 
write $_{B_1}\!M_{B_2}$ to emphasize the $(B_1,B_2)$-bimodule structure on $M$.
Now, if $B_1$ and $B_2$
are blocks of $kG_1$ and $kG_2$, respectively, then 
we can view every $(B_1,B_2)$-bimodule $M$ as a right $k(G_1\times G_2)$-module via the right 
$(G_1\times G_2)$-action defined 
by $m\cdot (g_1,g_2):={g_1}^{-1}mg_2$ for every $m\in M$, $g_1\in G_1$, $g_2\in G_2$. 
Furthermore, the blocks
$B_1$ and $B_2$ are called 
\emph{splendidly Morita equivalent} (or \emph{source-algebra equivalent}, or  
\emph{Puig equivalent}), if  
there is a Morita equivalence between $B_1$ and $B_2$ induced by a
$(B_1,B_2)$-bimodule $M$ which is 
is a $p$-permutation module when viewed as a right $k(G_1\times G_2)$-module. 
In this case, we write ${B_1}\sim_{SM}B_2$. 
By a result of Puig and Scott, this definition is equivalent to the condition that 
$B_1$ and $B_2$ have 
source algebras which are isomorphic as interior $P$-algebras 
(see \cite[Theorem~4.1]{Lin01}). Also, by a 
result of Puig (see \cite[Proposition~9.7.1]{Lin18}), the defect groups of 
$B_1$ and $B_2$ are isomorphic. 
Hence we may identify them.

\section{Finite groups with wreathed Sylow 2-subgroups}\label{sec:wreathedSylow} 

To begin with, we collect essential results about  finite groups 
with wreathed Sylow $2$-subgroups. 
In particular, we classify such groups modulo $O_{2'}(G)$. This classification is a byproduct of the
results of Alperin--Brauer--Gorenstein in \cite{ABG70}. 

\begin{Lemma}\label{lem:CGp5956}
Let $P:=C_{2^n}\wr C_2$ with $n\geq 2$.  Then the $2$-rank of $P$ is $2$ 
and $\Aut(P)$ is a $2$-group.
\end{Lemma}
\begin{proof} 
See e.g.~\cite[p.\,5956]{CG12}.
\end{proof}

For the benefit of legibility we state again Theorem~\ref{thm:classification_wr} 
of the introduction, before we prove it.

\begin{Theorem}\label{thm:classification_wr_bis} 
Let $G$ be a finite group with a Sylow $2$-subgroup
isomorphic to a wreathed $2$-group $C_{2^n}\wr C_2$ for an integer $n\geq 2$ 
such that $O_{2'}(G) = 1$. 
Then one of the following \smallskip holds: 
\begin{enumerate}[leftmargin=17mm]
\item[{\rm\sf(WR1)}] $G\cong
C_{2^n}\wr C_2$\,,
\item[{\rm\sf (WR2)}] $G
\cong (C_{2^n}\times C_{2^n})\rtimes \mathfrak S_3$\,,
\item[{\rm\sf(WR3)}]
$G\cong{\mathrm{SL}}{\color{black}^n_2}(q) \rtimes C_d$
where $q=r^{f}$ is a power of a prime number  $r$ with $f$ a positive integer, $(q-1)_2=2^n$ and $d$ is an odd integer such that  $d\mid f$\,;
\item[{\rm\sf(WR4)}]
$G\cong{\mathrm{SU}}{\color{black}^n_2}(q) \rtimes C_d$ where $q=r^{f}$ is a power of a prime number  $r$ with $f$ a positive integer,  $(q+1)_2=2^n$ and $d$ is an odd integer such that  $d\mid f$\,;
\item[{\rm\sf(WR5)}] $G\cong\PSL_3 (q).H$ where $q=r^{f}$ is a power of a prime number  $r$ with $f$ a positive integer,  $(q-1)_2=2^n$, 
$H \leq C_{(q-1, 3)}\times C_d$ and $d$ is an odd integer such that $d\mid f$\,; or
\item[{\rm\sf(WR6)}]
$G\cong\PSU_3(q).H$ where where $q=r^{f}$ is a power of a prime number  $r$ with $f$ a positive integer,  $(q+1)_2=2^n$, $H \leq
C_{(q+1, 3)}\times C_d$  and $d$ is an odd integer such that $d\mid f$\,.
\end{enumerate}
\end{Theorem}
\begin{proof}
If $G$ is $2$-nilpotent, then Case~{\rm\sf (WR1)} holds since $O_{2'}(G) = 1$. 
In all other cases, $G$ is a $D$-group, a $Q$-group or a $QD$-group with the notation of 
\cite[Definition~2.1]{ABG70}.
Let $G$ be a $D$-group. Then there exists $K\unlhd G$ of index $2$ such that 
$P\cap K\cong C_{2^n}\times C_{2^n}$. By \cite[Theorem~1]{Bra64}, 
$K\cong (C_{2^n}\times C_{2^n})\rtimes C_3$ and Case~{\rm\sf (WR2)} holds.

If $G$ is a $Q$-group, then Case {\rm\sf (WR3)} or {\rm\sf (WR4)} 
occurs by Propositions 3.2 and 3.3 of \cite{ABG70}.
Finally, let $G$ be a $QD$-group. Then by \cite[Proposition 2.2]{ABG70}, 
$N:=O^{2'}(G)$ is simple and the 
possible isomorphism types of $N$ are given by the main result of \cite{ABG73}. 
Since $C_G(N)\cap N=Z(N)=1$ we have $C_G(N)\le O_{2'}(G)=1$. 
The possibilities for $G/N\le\Out(N)$ can be deduced from \cite{Atlas}. 
Since $|G/N|$ is odd, no graph 
automorphism is involved. Hence, $G/N\le C_{(3,q-1)}\rtimes C_d$ or 
$G/N\le C_{(3,q+1)}\rtimes C_d$. In 
fact, $G/N$ must be abelian since $|G/N|$ is odd.
\end{proof}

\begin{Theorem}\label{thm:k(B)l(B)} 
Let $G$ be as in Theorem~\ref{thm:classification_wr_bis} and let $B:=B_0(kG)$. 
With the same labelling of 
cases as in Theorem~\ref{thm:classification_wr_bis} the following \smallskip holds:
\begin{enumerate}[leftmargin=20mm]
\item[{\rm\sf (WR1)\phantom{,4}}] $\ell(B)=1$, $k(B)=2^{2n-1}+3\cdot 2^{n-1}$, $k_0(B)=2^{n+1}$, 
$k_1(B)=2^{2n-1}-2^{n-1}$;
\item[{\rm\sf (WR2)\phantom{,4}}] $\ell(B)=2$, $k(B)=(2^{2n-1}+9\cdot 2^{n-1}+4)/3$, $k_0(B)=2^{n+1}$,\\
$k_1(B)=(2^{2n-1}-3\cdot 2^{n-1}+4)/3$; 
\item[{\rm\sf (WR3,4)}] $\ell(B)=2$, $k(B)=2^{2n-1}+2^{n+1}$, $k_0(B)=2^{n+1}$, 
$k_1(B)=2^{2n-1}-2^{n-1}$,\\ $k_n(B)=2^{n-1}$;
\item[{\rm\sf (WR5,6)}] $\ell(B)=3$, $k(B)=(2^{2n-1}+3\cdot 2^{n+1}+4)/3$, $k_0(B)=2^{n+1}$,\\ 
$k_1(B)=(2^{2n-1}-3\cdot 2^{n-1}+4)/3$, $k_{\color{black}n}(B)=2^{n-1}$.
\end{enumerate}
\end{Theorem} 

\begin{proof} 
Cases {\rm\sf (WR1)} and {\rm\sf (WR2)} follow from elementary group theory.
If Case {\rm\sf (WR3)} 
or Case {\rm\sf (WR4)} of Theorem~\ref{thm:classification_wr_bis} holds, then the numbers follow from 
\cite[Proposition~(7.G)]{Kue80}. 
Suppose now that Case {\rm\sf (WR5)} or Case {\rm\sf (WR6)} holds, then the 
number $k(B)$ follows from \cite[{\color{black}Theorem 1A}]{Brauerwr} -- 
here, Brauer even computed the degrees of the ordinary irreducible characters in $B$ -- 
whereas the number $\ell(B)$ can 
be obtained with \cite[Lemma~7.I]{Kue80} for instance.
\end{proof}

\section{Preliminaries}\label{sec:pre} 

We state below several results which will enable us to construct 
splendid Morita equivalences induced by Scott modules, but which are not restricted to characteristic $2$.
Therefore, throughout this section we may assume that $k$ is an algebraically 
closed field of arbitrary characteristic~$p>0$. \\

Our first main tool to construct splendid Morita equivalences is given by the 
following Theorem which is an extended version  
of a well-known result due to Alperin  \cite{Alp76} and Dade \cite{Dad77} 
 restated in terms of splendid Morita equivalences.

\begin{Theorem}[Alperin--Dade]\label{thm:AlperinDade}
Let $\widetilde{G}_{1}$ and $\widetilde{G}_{2}$ be a finite groups and assume $G_{1}\trianglelefteq \widetilde{G}_{1}$, $G_{2}\trianglelefteq \widetilde{G}_{2}$ 
are normal subgroups such that  $\widetilde{G}_{1}/G_{1}$, $\widetilde{G}_{2}/G_{2}$ are  $p'$-groups and having a common Sylow $p$-subgroup  $P\in\Syl_p(G_{1})\cap\Syl_{p}(G_{2})$ such that 
$\widetilde{G}_{1}=G_{1}C_{\widetilde{G}_{1}}(P)$ and $\widetilde{G}_{2}=G_{2}C_{\widetilde{G}_{2}}(P)$.  Then the following assertions hold. 
\begin{enumerate} 
\renewcommand{\labelenumi}{\rm{(\alph{enumi})}} 
\item
If $\tilde{e}$ and $e$ denote the block idempotents of $B_0(k\widetilde{G}_{1})$ 
and $B_0(kG_{1})$, respectively, then 
the map $B_0(kG_{1})\longrightarrow B_0(k\widetilde{G}_{1}), a\mapsto a\tilde{e}$ 
is an isomorphism of $k$-algebras. 
Moreover, the right $k[\widetilde{G}_{1}\times G_{1}]$-module
${\mathrm{Sc}}(\widetilde G_{1}\times G_{1},\Delta P) 
=B_0(k\widetilde G_{1}){\downarrow}^{\widetilde G_{1}\times\widetilde G_{1}}_{\widetilde{G}_{1}\times G_{1}}=\tilde{e}k\widetilde{G}_{1}=\tilde{e}k\widetilde{G}_{1}e$, 
induces a splendid Morita equivalence between $B_0(k\widetilde G_{1})$ and $B_0(kG_{1})$. 
\item
The Scott module ${\mathrm{Sc}}(\widetilde G_{1}\times\widetilde G_{2}, \Delta P)$ 
induces a splendid Morita equivalence 
between $B_0(k\widetilde G_{1})$ and $B_0(k\widetilde G_{2})$ if and only if  the Scott module 
${\mathrm{Sc}}(G_{1}\times G_{2}, \Delta P)$ induces a
splendid Morita equivalence 
between $B_0(kG_{1})$ and $B_0(kG_{2})$. 
\end{enumerate}
\end{Theorem} 
\begin{proof} 
Assertion~(a) follows from \cite{Alp76, Dad77}. More precisely, the given map 
is an isomorphism of $k$-algebras by \cite[Theorem]{Dad77} and 
\cite[Theorems 1 and  2]{Alp76} proves that restriction from $\widetilde{G}_{1}$ to $G_{1}$ induces a splendid Morita equivalence. 
Assertion~(b) is given by~\cite[Lemma~5.1]{KLS20}.\\
\end{proof}

\begin{Lemma}\label{lem:Aut(P)} 
Let $\widetilde{G}_{1}, \widetilde{G}_{2}$ be finite groups. Assume that 
$G_{1}\unlhd \widetilde G_{1}$ and  $G_{2}\unlhd \widetilde G_{2}$ are normal subgroups 
such that  $\widetilde{G}_{1}/G_{1}$, $\widetilde{G}_{2}/G_{2}$ are $p'$-groups 
and assume that $G_{1}$ and $G_{2}$ have a common Sylow $p$-subgroup  
$P$  such that 
$\Aut(P)$ is a $p$-group. 
Then, conclusions {\rm(a)} and {\rm(b)} of Theorem~\ref{thm:AlperinDade} hold. 
\end{Lemma}
\begin{proof}
It suffices to prove that the hypotheses of Theorem~\ref{thm:AlperinDade} are satisfied. So, let $i\in\{1,2\}$.
Since $\Aut(P)$ is a $p$-group we have  $N_{\widetilde G_{i}}(P)=PC_{\widetilde G_{i}}(P)$.
Moreover, by Frattini's argument
$\widetilde G_{i}=G_{i}N_{\widetilde G_{i}}(P)$, thus 
$\widetilde G_{i}=G_{i}C_{\widetilde G_{i}}(P)$, as required.
\end{proof}

Next, it is well-known that inflation from the quotient by a normal $p'$-subgroup induces an isomorphism of blocks as $k$-algebras.  In fact, there is splendid Morita equivalence induced by a Scott module and we have the following stronger result.

\begin{Lemma}\label{lem:G/Op'(G)}%
Let $G_{1}, G_{2}$ be finite groups with a common Sylow $p$-subgroup $P$. Let $N_{1}\unlhd G_{1}$ and $N_{2}\unlhd G_{2}$ be normal $p'$-subgroups and write $^{-}:G_{1}\longrightarrow G_{1}/N_{1}=:\overline{G_{1}}$, respectively {$^{-}:G_{2}\longrightarrow G_{2}/N_{2}=:\overline{G_{2}}$},  for the quotient homomorphisms, so that, by abuse of notation, we may identify  $\overline{P}=PN_{1}/N_{1}\cong P$ with $\overline{P}=PN_{2}/N_{2}\cong P$.   Then the following assertions hold: 
\begin{enumerate}[label={\rm(\alph*)}] 
\item
${\mathrm{Sc}}(G_{1}\times\overline{G_{1}}, \Delta P)$ induces 
a splendid Morita equivalence between $B_0(kG_{1})$ and $B_0(k\overline{G_{1}})$, where 
$\Delta P$ is identified with $\{(u, \bar u)\,|\, u\in P\}$;
\item
${\mathrm{Sc}}(G_{1}\times G_2,\Delta P)$ induces a splendid Morita equivalence
between $B_0(kG_{1})$ and $B_0(kG_{2})$ if and only if 
${\mathrm{Sc}}(\overline{G_{1}}\times\overline{G_{2}},\Delta\overline{P})$ induces
a splendid Morita equivalence 
between $B_0(k\overline{G_{1}})$ and $B_0(k\overline{G_{2}})$. 
\end{enumerate}
\end{Lemma}
\begin{proof} {\rm (a)} By the assumption $N_{1}\leq O_{p'}(G_{1})$, hence $N_{1}$ acts trivially on $B_0(kG_{1})$. 
Thus, $B_0(kG_{1})$ and its image $B_0(k\overline{G_{1}})$ in $k\overline{G_{1}}$ are isomorphic as
interior $P$-algebras. Part (a) follows then immediately from the fact that 
${\mathrm{Sc}}(G_{1}\times\overline G_{1}, \Delta P)={_{kG_{1}}B_0(kG_{1})}_{k\overline{G_{1}}}$ (seen as a $(kG_{1},k\overline{G_{1}})$-bimodule). Part {\rm (b)} follows from {\rm(a)} and the fact that
$${\mathrm{Sc}}(G_{1}\times\overline{G_{1}}, \Delta P)\,\otimes_{B_{0}(k\overline{G_{1}})} {\mathrm{Sc}}(\overline{G_{1}}\times \overline{G_2},\Delta \overline{P}) \,\otimes_{B_{0}(k\overline{G_{2}})} {\mathrm{Sc}}(\overline{G_{2}}\times G_{2}, \Delta P) 
 \cong  {\mathrm{Sc}}(G_{1}\times G_{2},\Delta P)\,.$$
 (See e.g. the proof of \cite[Lemma~5.1]{KLS20} for a detailed argument proving this isomorphism.)
\end{proof}

The following Lemma is also essential to treat central extensions.

\begin{Lemma}\label{lem:G/Z(G)}
Let $G_1$ and $G_2$ be finite groups having a common Sylow $p$-subgroup $P\in\Syl_p(G_{1})\cap\Syl_{p}(G_{2})$. 
Assume moreover that $Z_1\leq Z(G_1)$ and $Z_2\leq Z(G_2)$ are central subgroups 
such that $P\cap Z_1=P\cap Z_2$ (after identification of the chosen Sylow $p$-subgroups of $G_{1}$ and $G_{2}$). Set $\overline{G_1}:=G_1/Z_1$ and 
 $\overline{G_2}:=G_2/Z_2$. Then the subgroup  $\overline P:=PZ_1/Z_1 (\cong P/(P\cap Z_{1})\cong P/P\cap Z_{2})\cong  PZ_2/Z_2)$ 
can be considered as a common Sylow $p$-subgroup of $\overline{G_{1}}$ and $\overline{G_{2}}$.
Then, 
${\mathrm{Sc}}(G_1\times G_2,\Delta P)$ induces a splendid Morita equivalence
between $B_0(kG_1)$ and $B_0(kG_2)$ if and only if
${\mathrm{Sc}}(\overline{G_1}\times\overline{G_2},\Delta\overline P)$ induces
a splendid Morita equivalence 
between $B_0(k\overline{G_1})$ and $B_0(k\overline{G_2})$.
\end{Lemma}
\begin{proof}
Let $i\in\{1,2\}$.
Clearly, we have $Z_i=(P\cap Z_{i})\times O_{p'}(Z_i)$ and 
$$\overline{G_{i}}=G_{i}/Z_{i}\cong \left( G_{i}/(P\cap Z_{i}) \right) 
/\left( Z_{i} / (P\cap Z_{i}) \right)=:\overline{\overline{G_{i}}}\,.$$
Write $\widetilde{P}$  for the image of~$P$ in the quotients $G_{i}/(P\cap Z_{i})$ 
and write $\overline{\overline{P}}$ for the image of~$P$ in the quotients $\overline{\overline{G_{i}}}$. 
Now, on the one hand,  by Theorem~\ref{Prop3.3},  the Scott module 
{${\mathrm{Sc}}(G_1\times G_2,\Delta P)$} induces a splendid Morita equivalence
between $B_0(kG_1)$ and $B_0(kG_2)$
if and only if 
${\mathrm{Sc}}(G_1/(P\cap Z_{1})\times G_2/(P\cap Z_{2}),\Delta \widetilde{P})$ 
induces a splendid Morita equivalence
between $B_0(k[G_1/(P\cap Z_{1})])$ and $B_0(k[G_2/(P\cap Z_{2})])$, 
which by Lemma~\ref{lem:G/Op'(G)}(b)  happens if and only if 
${\mathrm{Sc}}(\overline{\overline{G_{1}}}\times \overline{\overline{G_{2}}},\Delta \overline{\overline{P}})$
induces a splendid Morita equivalence between $B_0(k\overline{\overline{G_1}})$ and $B_0(k\overline{\overline{G_2}})$.
The claim follows. 
\end{proof}

The next theorem is a standard method, called the ``gluing method'', which was already applied in 
\cite{KL20a,KLS20}.  It relies on gluing results, allowing us to 
construct stable equivalences of Morita type, and is a 
slight variation of different results of the same type due to 
Brou\'e, Rouquier, Linckelmann and Rickard.
See~e.g.~{\color{black}\cite[6.3.Theorem]{Bro94}}, \cite[Theorem~5.6]{Rou01} and 
\cite[Theorem 3.1]{Lin01}. 

\begin{Theorem}\label{Thm:ProveExSMEQ}
Let $G_1$ and $G_2$ be finite groups with a common Sylow $p$-subgroup $P$ 
satisfying $\mathcal F_P(G_1)=\mathcal F_P(G_2)$.  
Then, $M:={\mathrm{Sc}}(G_1\times G_2,\Delta P)$ 
induces a splendid Morita equivalence  between 
$B_0(kG_1)$ and $B_0(kG_2)$ provided the following two conditions are satisfied:
\begin{itemize}
\item[\rm(I)] for every subgroup $Q\leq P$ of order $p$, 
the bimodule $M(\Delta Q)$ 
induces a Morita equivalence between $B_0(k\,C_{G_1}(Q))$ and $B_0(k\,C_{G_2}(Q))$; and 
\item[\rm(II)] for every simple $B_0(kG_1)$-module $S_1$, the $B_0(kG_2)$-module 
$S_1\otimes_{B_0(kG_1)} M$ is again simple. 
\end{itemize}
\end{Theorem}
\begin{proof}
By \cite[Lemma~4.1]{KL20a}, Condition~(I) is equivalent to the fact that $M$
 induces a stable equivalence of Morita type
between $B_0(kG_1)$ and $B_0(kG_2)$. Therefore, applying  
\cite[Theorem 2.1]{Lin96a}, Condition~(II) now implies that $M$ 
induces a Morita equivalence between $B_0(kG_1)$ and $B_0(kG_2)$. 
This equivalence is necessarily splendid 
since $M$ is a $p$-permutation module by definition.  
\end{proof}

\begin{Lemma}\label{lem:simpleSocle}
Let $A$ be a finite-dimensional $k$-algebra. Let $X$ be an $A$-module and 
let $Y$ be an $A$-submodule such that $X/Y$ and $\soc(Y)$ are both simple.
If $Y$ is not a direct summand of $X$, then $\soc(X)=\soc(Y)$, and hence $X$ is indecomposable. 
\end{Lemma}
\begin{proof}
Since $\soc(X)$ and $\soc(Y )$ are semisimple, we have $\soc(X)\cap Y = \soc(Y )$ and $\soc(X) = \soc(Y)\oplus S$ where $S$ is a submodule of $X$. Thus $S\cap Y = S\cap \soc(X)\cap Y = S\cap \soc(Y) = 0$. Hence, either $S = 0$ and $\soc(X) = \soc(Y)$, or $Y$ is a submodule of  $S\oplus Y$ and  so $S\oplus Y =X$.
\end{proof}

Finally, the next lemma is often called the ``stripping-off method''. It will be used to verify Condition (II) of Theorem~\ref{Thm:ProveExSMEQ}
in concrete cases.

\begin{Lemma}[{\cite[Lemma A.1]{KMN11}}]\label{StrippingOffMethod}
Let $A$ and $B$ be self-injective finite-dimensional $k$-algebras. 
Let $F: \mod{\text{-}}A\longrightarrow \mod{\text{-}}B$ be a covariant functor satisfying the following conditions:
 \begin{enumerate}
  \renewcommand{\labelenumi}{\rm{(C\arabic{enumi})}}
   \item $F$ is exact;
   \item if $X$ is a projective $A$-module, then $F(X)$ is a 
    projective $B$-module;
   \item $F$ realises a stable equivalence from
    $\mod{\text{-}}A$ to $\mod{\text{-}}B$.
 \end{enumerate}
Then, the following assertions hold.
\begin{enumerate}

 \item[\rm(a)]
{\sf (Stripping-off method, case of socles.)} 
  Let $X$ be a projective-free $A$-module, and write
  $F(X) = Y \oplus ({\mathrm{proj}})$ where $Y$ is a projective-free $B$-module. 
  Let $S$ be a simple $A$-submodule of $X$ and set $T := F(S)$. 
  If $T$ is a simple non-projective $B$-module, then there exists a $B$-submodule $W$ of $F(X)$
  such that $W\cong Y$,  $T\subseteq W$ and 
  \[
   F(X/S)\cong W/T \oplus ({\mathrm{proj}})\,.
   \]
 \item[\rm(b)] 
{\sf (Stripping-off method, case of radicals.)} 
  Let $X$ be a projective-free $A$-module, and write $F(X) = Y \oplus R$ where $Y$ is  a projective-free $B$-module and $R$ is a projective $B$-module. 
  Let $X'$ be an $A$-submodule of $X$ such that $X/X'$ is simple and 
  let ${\pi:X\longrightarrow X/X'}$ be the quotient homomorphism.
  If $T:= F(X/X')$ is a simple $B$-module,
  then there exists a $B$-submodule $R'$ of $F(X)$ such that $R'\cong R$, $R'\subseteq \ker(F(\pi))$, $F(X) = Y \oplus R'$
  and 
  $${\mathrm{ker}}\left(F(X) \overset{F(\pi)}{\twoheadrightarrow} F(X/X')\right) =  
{\mathrm{ker}}\left(Y \overset{F(\pi)|_{Y}}{\twoheadrightarrow }F(X/X')   \right) \oplus   ({\mathrm{proj}})\,.  $$
\end{enumerate}
\end{Lemma}

\section{Groups of type {\sf (W3(n))} and {\sf (W4(n))}}
\label{sec:W3W4}

\begin{Hypothesis}\label{Ass:p=2}%
From now on and until the end of this manuscript we assume that
the algebraically closed field $k$ has characteristic $p=2$. 
Furthermore, $G$, $G_{1}$, $G_{2}$, $\mathcal{G}_{1}$, $\mathcal{G}_{2}$, $\mathsf{G}$, $\mathsf{G}_{1}$ and $\mathsf{G}_{2}$ always denote finite groups 
with a common Sylow $2$-subgroup $P\cong C_{2^n}\wr C_2$, 
where $n\geq 2$ is a fixed integer. In other words, we choose a Sylow $2$-subgroup of each of these groups and we identify them for simplicity. 
Moreover,  $q$, $q_{1}$ and $q_{2}$ are (possibly different) positive powers of odd prime numbers.
\end{Hypothesis}

In this section and the next two ones, we prove Theorem~\ref{thm:MainResult}(b) 
through a case-by-case analysis. 
We start with the groups of types {\sf (W3(n))} and {\sf (W4(n))}, 
for which we reduce the problem to the classification of principal blocks 
with dihedral defect groups up to splendid Morita equivalence obtained in 
\cite[Theorem~1.1]{KL20a}.  
The group theory setting to keep in mind is described in the following remark.

\begin{Remark}\label{GL-GU}%
For any positive power $q$ of an odd prime number \cite[p.4]{ABG70} shows that we have the 
following inclusions of normal subgroups with the given indices:
\[
\begin{tikzcd}
   \GL_2(q) \ar[d, dash, "(q-1)/2^n"']  & &   \GU_2(q) \ar[d, dash, "(q+1)/2^n"]   \\
 \SL_2^n(q) \ar[d, dash, "2^{n}"']   &   &   \SU_2^n(q)  \ar[d, dash, "2^{n}"]   \\
  \SL_2(q)   &  \cong   &   \SU_2(q) \\
\end{tikzcd}
\]
\end{Remark}

\begin{Proposition}\label{pro:SL2^m} 
For each $i\in\{1,2\}$ let $G_i:=\SL^n_2(q_i)$, 
$\mathcal G_i:=\GL_2(q_i)$ and assume that {$(q_i-1)_2=2^n$}.
Then, the following assertions hold:
\begin{enumerate}[label={\rm(\alph*)}, leftmargin=10mm] 
\item   
${\mathrm{Sc}}(\mathcal{G}_1\times\mathcal{G}_2, \Delta P)$ induces 
a splendid Morita equivalence between $B_0(k\mathcal G_1)$ and $B_0(k\mathcal G_2)$;
\item
 ${\mathrm{Sc}}(G_1\times G_2, \Delta P)$\! 
induces a splendid Morita equivalence between $B_0(kG_1)$\! and $B_0(kG_2)$.
\end{enumerate}
\end{Proposition}

\begin{proof}
Elementary calculations yield $G_{i}\lhd\mathcal{G}_{i}$ and $|\mathcal G_i/G_i|=({q_i}-1)/2^n$ 
for each $i~\in~\{1,2\}$  (see~Remark~\ref{GL-GU}). In particular 
both indices are odd. Hence, by Lemma~\ref{lem:CGp5956} and Lemma~\ref{lem:Aut(P)}, 
assertion (b) follows from assertion (a), so it suffices to prove (a). \par
Now, $P\cap Z(\mathcal G_1)=P\cap Z(\mathcal G_2)=Z(P)$, so 
$\overline P:=(PZ(\mathcal G_1))/Z(\mathcal G_1)\cong (PZ(\mathcal G_2))/Z(\mathcal G_2)\,,$
and hence, up to identification, we can consider that 
$\overline P\in\Syl_2(\mathcal G_1/Z(\mathcal G_1))\cap\Syl_2(\mathcal G_2/Z(\mathcal G_2))$.
Moreover, we have 
$$\overline P\cong P/Z(P)\cong D_{2^{n+1}}\,,$$
see e.g. \cite[(2.A)~Lemma~(iii)]{Kue80}.
Since $\mathcal G_i/Z(\mathcal G_i)\cong \PGL_2(q_i)$ for each $i\in\{1,2\}$, assertion (a) now follows
directly  from Lemma~\ref{lem:G/Z(G)} and \cite[Theorem 1.1]{KL20a}.
\end{proof}

\begin{Proposition}\label{pro:SU2^m}%
For each $i\in\{1,2\}$ let $G_i:=\SU^n_2(q_i)$, $\mathcal G_i:=\GU_2(q_i)$ and assume that $(q_i+1)_2=2^n$. Then, the following assertions hold:
\begin{enumerate}[label={\rm(\alph*)}, leftmargin=10mm] 
\item   
${\mathrm{Sc}}(\mathcal{G}_1\times\mathcal{G}_2, \Delta P)$ induces 
a splendid Morita equivalence between $B_0(k\mathcal G_1)$ and $B_0(k\mathcal G_2)$;
\item
 ${\mathrm{Sc}}(G_1\times G_2, \Delta P)$\! 
induces a splendid Morita equivalence between $B_0(kG_1)$\! and $B_0(kG_2)$.
\end{enumerate}
\end{Proposition} 
\begin{proof}
In this case $G_{i}\lhd\mathcal{G}_{i}$ and 
$|\mathcal G_i/G_i|=(q_i+1)/2^n$ for each $i\in\{1,2\}$ (See~Remark~\ref{GL-GU}). 
Thus both indices are odd. 
Again by Lemma~\ref{lem:CGp5956} and Lemma~\ref{lem:Aut(P)}, it suffices to prove (a).
\par
Now, $P\cap Z(\mathcal G_1)=P\cap Z(\mathcal G_2)=Z(P)$. Thus 
$\overline P:=(PZ(\mathcal G_1))/Z(\mathcal G_1)
\cong (PZ(\mathcal G_2))/Z(\mathcal G_2)$
and we can consider that 
$\overline P\in\Syl_2(\mathcal G_1/Z(\mathcal G_1))\cap\Syl_2(\mathcal G_2/Z(\mathcal G_2))$.
As in the previous proof, 
$$\overline P\cong P/Z(P)\cong D_{2^{n+1}}\,.$$  
Next,  for each for $i\in\{1,2\}$ we have an isomorphism $\SU_2(q_i)\cong\SL_2(q_i)$,  
and hence $\PSU_2(q_i)\cong\PSL_2(q_i)$. Furthermore, since $q_i$ is odd, 
$\PGL_2(q_i)=\PSL_2(q_i).2$ (where $2$ denotes the cyclic group of order $2$ generated by
the diagonal automorphism of $\PSL_2(q_i)$) by Steinberg's result
(see \cite[Chap.6 (8.8), p.\,511 and Theorem 8.11]{Suz86}). 
In other words, we have
\begin{equation}\label{PGL-PGU} \qquad
\mathcal G_i/Z(\mathcal G_i)=\PGU_2(q_i)\cong\PGL_2(q_i)\,.
\end{equation}
Therefore, assertion (a) follows immediately from Lemma~\ref{lem:G/Z(G)}
and~\cite[Theorem~1.1]{KL20a}, proving the Proposition.
\end{proof}


\section{Groups of type {\sf (W5(n))}}
\label{sec:W5}

We now turn to the groups of type {\sf (W5(n))}.  
We continue using Hypothesis~\ref{Ass:p=2}.

\begin{Notation}\label{nota:PSL_{3}}%
Throughout this section we let  $i\in\{1,2\}$ be arbitrary and set  $G_{i}:=\PSL_{3}(q_{i})$, $\mathsf G_{i}:=\SL_{3}(q_{i})$ and $\widetilde{G}_{i}:=\GL_{3}(q_{i})$ where we assume that $(q_{i}-1)_{2}=2^{n}$. After identification, we may assume that  $G_{1}$, $G_{2}$, $\mathsf G_{1}$ and $\mathsf G_{2}$ have a common Sylow $2$-subgroup $P$ isomorphic to $C_{2^n}\wr C_2$. 
Then,
\begin{equation}\label{sm:PSLSLGL}
B_{0}(kG_{i})   \sim_{SM}   B_{0}(k\mathsf{G}_{i}) 
\end{equation}
where the splendid Morita equivalence is induced by inflation (as $Z(\SL_{3}(q_{i}))\cong C_{(3,q_{i}-1)}$ is a $2'$-group).
Using~\cite[Proposition~4.3.1 and Remark~4.2.1]{GM20} we know that $B_{0}(k\mathsf{G}_{i})$ contains three unipotent characters, namely
$$1_{\mathsf{G}_{i}}, \chi^{}_{q_{i}^2+q_{i}}, \chi^{}_{q_{i}^3}\,,$$ 
where we use the convention that the indices denote the degrees, whereas those lying in $B_{0}(k\widetilde{G}_{i})$ can be written as 
$$1_{\widetilde{G}_{i}}, \widetilde \chi^{}_{q_{i}^2+q_{i}}, \widetilde \chi^{}_{q_{i}^3}\,$$ 
and satisfy $1_{\widetilde{G}_{i}}{\downarrow}_{\mathsf{G}_{i}}=1_{\mathsf{G}_{i}}$,  $\widetilde \chi^{}_{q_{i}^2+q_{i}}{\downarrow}_{\mathsf{G}_{i}}=\chi^{}_{q_{i}^2+q_{i}}$ and $\widetilde \chi^{}_{q_{i}^3}{\downarrow}_{\mathsf{G}_{i}}=\chi^{}_{q_{i}^3}$.  (We also refer to \cite{Ste51}, \cite[7.19. Theorem(i)]{Jam86}, that first described these characters and their degrees.)\\
We obtain from \cite[\S4]{Jam90} and the above that $3=\ell(B_{0}(k\widetilde{G}_{i}))=\ell(B_{0}(k\mathsf{G}_{i}))$ and we may write
$$\Irr_{k}(B_{0}(k\mathsf{G}_{i}))=:\{k_{\mathsf{G}_{i}}, S_{i}, T_{i}\}\quad\text{ and }\quad \Irr_{k}(B_{0}(k\widetilde{G}_{i}))=:\{k_{\widetilde{G}_{i}}, \widetilde{S}_{i}, \widetilde{T}_{i}\}\,,$$
where  $S_i=\widetilde S_i{\downarrow}_{\mathsf{G}_{i}}$ and $T_i=\widetilde T_i{\downarrow}_{\mathsf{G}_i}$.
Moreover, by \cite[p.\,253]{Jam90},  the part of the $2$-decomposition matrix of $B_{0}(k\widetilde{G}_{i})$ whose rows are labelled by 
the unipotent characters is as follows:
\[
\begin{array}{l|ccc} & k_{\widetilde{G}_{i}} & \widetilde S_{i} &  \widetilde T_{i}\\ \hline
1_{\widetilde{G}_{i}} & 1& .&. \\ 
\widetilde\chi^{}_{q_{i}^2+q_{i}}& . &1 & .\\
\widetilde\chi^{}_{q_{i}^3} & 1& . & 1 \\
\end{array}
\]
(This is the case $\Delta_3$ with $n=3$, $e=2$ and $p\geq 2$.)
\end{Notation}

We start by describing some trivial source modules belonging to the principal $2$-block of $\SL_{3}(q_{i})$ which we will use in the sequel. 

\begin{Lemma}\label{lem:tsPSL_{3}(q)}%
The principal block $B_{0}(k\mathsf{G}_{i})$ contains, amongst others, the following trivial source modules:
\begin{enumerate}
\item[\rm(a)] the trivial module $k_{\mathsf{G}_{i}}$, with vertex $P$ and affording the trivial character $1_{ \mathsf{G}_{i} }$\,;
\item[\rm(b)] the simple module~$S_{i}$, having  $Q:=C_{2^n}\times C_{2^n}\leq P$ as a vertex, and affording the character $\chi^{}_{q_{i}^2+q_{i}}$\,;
\item[\rm(c)] the Scott module $\Sc(\mathsf{G}_{i},Q)$ with vertex $Q$, satisfying $Sc(\mathsf{G}_{i},Q)\ncong S_{i}$;
\item[\rm(d)] the Scott module $\Sc(\mathsf{G}_{i},\mathbb B_{i})$ on a Borel subgroup $\mathbb B_{i}$ of $\mathsf G_{i}$,
which is uniserial with composition series 
$$
\boxed{\begin{matrix}k_{\mathsf{G}_{i}}\\ T_{i}\\ k_{\mathsf{G}_{i}}\end{matrix}}
$$ 
and affords the character $1_{\mathsf{G}_{i}}+\chi_{q_{i}^3}$.
\end{enumerate}
\end{Lemma}

\begin{proof}
First we note that it is clear that all the given modules belong to the principal block as at least one of their constituents obviously does. \\
{\rm(a)} 
It is clear that the trivial module is a trivial source module with vertex $P$ affording the trivial character.\\
{\rm(b)}  
As the restriction of a trivial source module is always a trivial source module, to prove that $S_{i}$ is a trivial source module affording $\chi^{}_{q_{i}^2+q_{i}}$, it is enough to prove that the 
$k\widetilde{G}_{i}$-module $\widetilde{S}_{i}$ is a trivial source module affording  $\widetilde\chi^{}_{q_{i}^2+q_{i}}$. (See e.g. \cite[\S4]{Las23} for these properties.)
Now,  \cite[pp.\,228--229]{Ste51} shows that $1_{\widetilde{G}_{i}} + \widetilde\chi_{q^2+q}$ is a permutation character. More precisely
there exists a subgroup $\widetilde H_{i}\leq \widetilde G_{i}$
such that $\widetilde H_{i}\cong (C_{q_{i}}\times C_{q_{i}})\rtimes\GL_2(q_{i})$, $|\widetilde G_{i}:\widetilde H_{i}|=1+q_{i}+q_{i}^2$ and 
${1_{\widetilde H_{i}}}{\uparrow}^{\widetilde G_{i}}= 1_{\widetilde{G}_{i}} + \widetilde\chi_{q_{i}^2+q_{i}}$\,.
Thus, setting $X_{i}:=k_{\widetilde H_{i}}{\uparrow}^{\widetilde G_{i}}$, the decomposition matrix given in Notation~\ref{nota:PSL_{3}} implies that
$$X_{i} = k_{\widetilde G_{i}}+\widetilde S_{i} \text{ (as composition factors)}\,.$$ 
Then $X_{i}=k_{\widetilde G_{i}}\oplus \widetilde S_{i}$  as $k_{G_{i}}$ must 
occur as a composition factor of the socle and of the head, proving that $\widetilde S_{i}$ is a trivial source module affording the character $\widetilde\chi^{}_{q_{i}^2+q_{i}}$.  
Finally, using  \cite[II Lemma 12.6(iii)]{Lan83} and the character table of $\SL_{3}(q_{i})$ in~\cite{SF73} we can read from the values of the character 
$\chi^{}_{q_{i}^2+q_{i}}$ at non-trivial $2$-elements that $Q=C_{2^n}\times C_{2^n}\leq P$ is a vertex of $S_i$. \\
{\rm(c)} 
The Scott module $\Sc(\mathsf{G}_{i},Q)$ is a trivial source module with vertex~$Q$ 
and clearly $S_{i}\ncong \Sc(\mathsf{G}_i,Q)$, as a Scott module always has 
a trivial constituent in its head by definition.\\
{\rm(d)} 
\cite[pp.\,228--229]{Ste51} also shows that  $1_{\widetilde{G}_{i}}+2\widetilde\chi_{q_{i}^2+q_{i}}+\widetilde\chi_{q_{i}^3}$ is a permutation character.
More precisely,  there is a Borel subgroup $\widetilde{\mathbb B}_{i}\leq\widetilde G_{i}$  such that
$1_{\widetilde{\mathbb B}_{i}}{\uparrow}^{\widetilde G} 
= 1_{\widetilde{G}_{i}} +2\widetilde\chi_{q_{i}^2+q_{i}}+\widetilde\chi_{q_{i}^3}$.
Setting 
$Y_{i}:={k_{\widetilde{\mathbb B}_{i}}}{\uparrow}^{\widetilde G_{i}}$ we obtain from the decomposition matrix in Notation~\ref{nota:PSL_{3}}  that 
$$Y_{i} = 2\times k_{\widetilde G_{i}} + 2\times\widetilde S_{i}+ \widetilde T_{i}\quad \text{(as
composition factors).}$$
As both $Y_{i}$ and $\widetilde S_{i}$ are trivial source modules, we have
$$\dim_{k}\,\Hom_{k\widetilde {G}_{i}}(Y_{i},\widetilde {S}_{i})
=\dim_{k}\,\Hom_{k\widetilde{G}_{i}}(\widetilde{S}_{i}, Y_{i})
=\langle 1_{\widetilde G_{i}}  +2\widetilde\chi_{q_{i}^2+q_{i}}
+\widetilde\chi_{q_{i}^3}, \widetilde\chi_{q_{i}^{2}+q_{i}} \rangle^{}_{\widetilde G_{i}} = 2$$
(see~\cite[II Theorem 12.4(iii)]{Lan83}), implying that  $\widetilde S_{i}\oplus  \widetilde  S_{i}\mid \soc(Y_{i})$ and $\widetilde  S_{i}\oplus \widetilde  S_{i}\mid  \hd(Y_{i})$. 
Thus, there exists a submodule $U_{i}$ of $Y_{i}$ such that  $Y_{i}\cong \widetilde S_{i}\oplus \widetilde S_{i}\oplus U_{i}$
and hence $U_{i}$ is a trivial source module with composition factors
$2\times k_{\widetilde G_{i}}+T_{i}$ and $U_{i}$ affords the ordinary character $1_{\widetilde{G}_{i}}+\widetilde\chi_{q_{i}^3}$. Applying~\cite[II Theorem 12.4(iii)]{Lan83} again, we get 
$$\dim_k\,\Hom_{k\widetilde G}(U_{i},U_{i})=\langle 1_{\widetilde{G}_{i}}+\widetilde\chi_{q_{i}^3},  1_{\widetilde{G}_{i}}+\widetilde\chi_{q_{i}^3}  \rangle^{}_{\widetilde G_{i}} =2$$
and 
$$\dim_k\,\Hom_{k\widetilde G_{i}}(k_{\widetilde G_{i}}, U_{i})=\langle 1_{\widetilde{G}_{i}}, 1_{\widetilde{G}_{i}}+\widetilde\chi_{q_{i}^3}  \rangle^{}_{\widetilde G_{i}} =1
=\dim_k\,\Hom_{k\widetilde G}(U_{i}, k_{\widetilde G_{i}})\,.$$
It follows that 
\begin{equation*}\label{Scott}
U_{i} =
\boxed{\begin{matrix} k_{\widetilde G_{i}}\\ \widetilde T \\ k_{\widetilde G_{i}}
\end{matrix}}
= \Sc(\widetilde G_{i}, \widetilde{\mathbb B}_{i})
\end{equation*}
and setting ${\mathbb B}_{i}:=\widetilde{\mathbb B}_{i}\cap \mathsf{G}_{i}$ yields assertion (d).
\end{proof}

\noindent We can now prove Theorem~\ref{thm:MainResult}(b) for  the groups of types {\sf (W5(n))}. 

\begin{Proposition}\label{prop:PSL3Puig} 
The Scott module ${\mathrm{Sc}}(G_1\times G_2, \Delta P)$ induces a splendid 
Morita equivalence between $B_0(kG_1)$ and $B_0(kG_2)$. 
\end{Proposition}
\begin{proof} Below $i\in\{1,2\}$. First, we observe that by Lemma~\ref{lem:G/Z(G)}, ${\mathrm{Sc}}(G_1\times G_2, \Delta P)$ induces a splendid 
Morita equivalence between the principal blocks $B_0(kG_1)$ and $B_0(kG_2)$ if and only if 
$\Sc(\mathsf{G}_1\times \mathsf{G}_2, \Delta P)=:M$ induces a splendid Morita equivalence between 
$B_{1}:=B_0(k\mathsf{G}_1)$ and $B_{2}:=B_0(k\mathsf{G}_2)$. Thus, we may work with $\mathsf{G}_{i}$ instead of $G_{i}$ (for $i\in\{1,2\}$). 
Now, observe that 
$\mathcal F_P(\mathsf{G}_1)=\mathcal F_P(\mathsf{G}_2)$  and 
all involutions in $\mathsf{G}_i$ are $\mathsf{G}_i$-conjugate (see e.g. \cite[Theorem~5.3]{CG12} 
and \cite[Proposition~2 on p.11]{ABG70}).  Thus, it follows that it now suffices to prove that Conditions
(I) and (II) of Theorem~\ref{Thm:ProveExSMEQ} \smallskip hold.\\
{\bf Condition (I)}. 
By the above we only need to consider one involution in $P$, so we choose an involution $z\in Z(P)$,
and set $C_i:=C_{\mathsf{G}_i}(z)$. Clearly, $C_{i}\cong\GL_2(q_i)$ and again, 
up to identification, we see $P\in\Syl_2(\mathcal G_1)\cap\Syl_2(\mathcal G_2)$ (see~Remark~\ref{GL-GU}). 
We have to prove that $M(\Delta\langle z\rangle)$ 
induces a Morita equivalence between $B_0(kC_1)$ and $B_0(kC_2)$.  
Now, recall that $M_{z}:=\Sc(C_1\times C_2,\Delta P)$ induces a splendid Morita equivalence between $B_0(kC_1)$ and $B_0(kC_2)$
by Proposition~\ref{pro:SL2^m}(a). 
Moreover, obviously, it is always true that $M_z\,|\,M(\Delta\langle z\rangle)$,
and  we obtain that equality $M_z=M(\Delta\langle z\rangle)$  holds by the Brauer indecomposability 
of $M$ proved in ~\cite[Theorem~1.1]{KT20}. 
Thus Condition (I) is \smallskip verified.\\
{\bf Condition (II)}. We have to prove that the functor $-\otimes_{B_1}M$ maps the simple 
$B_{1}$-modules to the simple $B_{2}$-modules. 
First, we have $k_{\mathsf G_1}\otimes_{B_1}M\cong k_{\mathsf G_2}$ by \cite[Lemma~3.4(a)]{KL20a}.
Next, as  $N_{\mathsf{G}_i}(Q)/Q\cong\mathfrak S_3$, there are precisely $|\mathfrak{S}_{3}|_{2}=2$ non-isomorphic 
trivial source $k\mathsf{G}_{i}$-modules (see e.g. \cite[Theorem~4.6(c)]{Las23}), namely the modules $\Sc(\mathsf{G}_i,Q)$ and $S_i$, both belonging to the principal block by Lemma~\ref{lem:tsPSL_{3}(q)}.
Now, on the one hand, we know  from \cite[Theorem 2.1(a)]{KL20a} that  $S_1\otimes_{B_1}M=:V$ is indecomposable and non-projective, and on the other hand we know from \cite[Lemma 3.4(b)]{KL20a}
that $V$ is a trivial source module with vertex $Q$. Thus $V$ is either $\Sc(\mathsf{G}_2,Q)$ or $S_2$. However, $\Sc(\mathsf{G}_{1},Q)\otimes_{B_1}M\cong \Sc(\mathsf{G}_{2},Q)\oplus\text{({\sf proj})}$ by \cite[Lemma 3.4(c)]{KL20a}.
Hence, it follows immediately  that 
$$S_1\otimes_{B_1}M\cong S_2\,. $$
It remains to treat $T_{1}$. By our assumption, $(q-1)_2=(q_2-1)_2=2^n$, so the Sylow $2$-subgroups of 
$\mathbb B_{1}$ and $\mathbb B_2$ are isomorphic, meaning that the Scott modules $\Sc(\mathsf{G}_1, \mathbb B_{1})$ and $\Sc(\mathsf{G}_2, \mathbb B_{2})$
have isomorphic vertices (see e.g.~\cite[Corollary 4.8.5]{NT89}). Therefore, \cite[Lemma 3.4(c)]{KL20a} together with Lemma~\ref{lem:tsPSL_{3}(q)}(d) yield
$$
\boxed{\begin{matrix}k_{\mathsf{G}_1}\\ T_{1}\\ k_{\mathsf{G}_1}\end{matrix}}\otimes_{B_1}M 
= \Sc(\mathsf{G}_1, \mathbb B_{1})\otimes_{B_1}M
\cong \Sc(\mathsf{G}_2, \mathbb B_{2}) \oplus\text{({\sf proj})}
= \boxed{\begin{matrix}k_{\mathsf{G}_2}\\ T_{2}\\ k_{\mathsf{G}_2}\end{matrix}} \oplus\text{({\sf proj})}
$$
and Lemma~\ref{StrippingOffMethod} implies that  $T_1\otimes_{B_1}M\cong T_2  \oplus\text{({\sf proj})}$\,. However, again \cite[Theorem~2.1(a)]{KL20a} tells us that $T_1\otimes_{B_1}M$ is indecomposable non-projective, proving that
$$T_1\otimes_{B_1}M\cong T_2\,. $$
Thus, Condition (II) is verified and the proposition is proved. 
 \end{proof}


\section{Groups of type {\sf (W6(n))}}\label{sec:W6}

Finally, we examine the groups of type {\sf (W6(n))}, and we continue using Hypothesis~\ref{Ass:p=2}. 
Our aim is to prove Theorem~\ref{thm:MainResult}(b) for such groups.  
However, in order to reach this aim, first we start by collecting some information about the principal $2$-block of $\PGU_{3}(q)$ 
and about some of its modules.

\begin{Notation}\label{nota:PGU}%
Throughout this section, given a positive power $q$ of prime number satisfying $(q+1)_{2}=2^{n}$, 
we set the following notation.
The $3$-dimensional projective unitary group is 
$$\GU_{3}(q)=\{(a_{rs})\in \GL_{3}(q^{2})\mid (a_{sr})w_{0}(a_{rs}^{q})=w_{0}\}\qquad\text{ with }\quad w_{0}:=\left(\begin{smallmatrix} 0& 0&1 \\ 0& 1&0 \\ 1&0 &0
\end{smallmatrix}\right),$$
$\mathsf{G}:=\mathsf{G}(q):=\PGU_{3}(q)=\GU_{3}(q)/Z(\GU_{3}(q))$ where $Z(\GU_{3}(q))$ consists of the scalar matrices in $\GU_{3}(q)$, 
and $\PSU_{3}(q)=:G(q)$ is the commutator subgroup of $\PGU_{3}(q)$, which is a normal subgroup of index $(3,q+1)$.
Furthermore, we let $\mathbb{B}:=\mathbb{B}(q)$ denote the Borel subgroup of $\GU_{3}(q)$ defined by $\mathbb{B}(q)=\mathbb{T}(q)\mathbb{U}(q)$, with
$\mathbb{T}(q):=\{\text{diag}(\zeta^{-1},1,\zeta^{q})\mid \zeta\in\mathbb{F}_{q^{2}}^{\times}\}$ and 
$$
\mathbb{U}(q):=\left\{ 
 \left( \begin{smallmatrix} 1 & 0&0 \\   \alpha &1 &0 \\  \beta & -\alpha^{q}&1 
  \end{smallmatrix}\right)
   \in\GU_{3}(q)\mid 
   \alpha,\beta\in\mathbb{F}_{q^{2}}\text{ and }\alpha^{q+1}+\beta^{q}+\beta=1
  \right\}.$$
It is clear that $\mathbb{B} \cap Z(\GU_{3}(q))=1$, thus we may, and we do, identify $\mathbb{B}$ with a subgroup of $\PGU_{3}(q)$.\\
Next, we observe that~\cite[Theorem 1A]{Brauerwr} gives us the number of ordinary characters in the principal $2$-block of $G$ and their degrees. 
Moreover, using \cite{SF73} and \cite[Table~1.1 and Table~3.1]{Gec90}, or CHEVIE \cite{Chev} it is easy to compute central characters and we have that $B_{0}(k\mathsf{G})$ contains the following ordinary irreducible characters, 
in the notation of \cite{Gec90}: 
$$
\begin{array}{l|c|c} & \mathrm{condition}&\mathrm{number \ of \ characters}\\ 
\noalign{\hrule }
1_{\mathsf{G}}   &  &1 \\ 
\chi_{q(q-1)}&  &1 \\
\chi_{q^3}&  &1 \\
\hline
\chi_{q^2-q+1}^{(u)}&  u\equiv 0\pmod{(q+1)_{2'}}&2^n-1  \\ 
\hline
\chi_{q(q^2-q+1)}^{(u)}& u\equiv 0\pmod{(q+1)_{2'}}&2^n-1  \\
\hline
\chi_{(q-1)(q^2-q+1)}^{\color{black}(u,v)} &u,v\equiv 0\pmod{(q+1)_{2'}} & (2^n-1)(2^{n-1}-1)/3 \\
\hline
\chi_{q^3+1}^{(u)}& u\equiv 0\pmod{(q+1)_{2'}}&2^{n-1}
\end{array}
$$
where the subscripts denote the degrees.
Finally, the principal block of $k\mathsf{G}$ contains precisely three pairwise non-isomorphic  simple modules and we write
$$\Irr_{k}(B_{0}(k\mathsf{G}))=\{k_{\mathsf{G}}, \varphi, \theta \}$$
as  in \cite[Theorem~4.1]{His04} where the simples and their Brauer characters are identified for simplicity.

\end{Notation}

\begin{Lemma}\label{lem:2decPSU3}
With the notation of Notation~\ref{nota:PGU}, the decomposition matrix of the principal $2$-block of $\mathsf{G}=\PGU_3(q)$ is as follows:
$$
\begin{array}{l|ccc|c|c} & k_{\mathsf{G}} &\varphi &\theta &\mathrm{number \ of \ characters}\\ 
\noalign{\hrule }
1_{\mathsf{G}}   & 1& . & . & 1 \\ 
\chi_{q(q-1)}& . & 1&.& 1 \\
\chi_{q^3}&1& 2&1& 1 \\
\hline
\chi_{q^2-q+1}^{(u)}&1&1& . & 2^n-1 \\ 
\hline
\chi_{q(q^2-q+1)}^{(u)}&1&1&1&2^n-1 \\
\hline
\chi_{(q-1)(q^2-q+1)}^{\color{black}(u,v)}& . & . &1& (2^n-1)(2^{n-1}-1)/3\\
\hline
\chi_{q^3+1}^{(u)}&2&2&1& 2^{n-1}
\end{array}
$$
\end{Lemma}
\begin{proof}
To start with,~\cite[Appendix]{His04} gives us the unipotent part of the decomposition matrix. (See also \cite[Table 4.5]{GJ11}.)
Then, direct computations using \cite[Table 1.1 and Table 3.1]{Gec90} (see also \cite{SF73}) or {\sf CHEVIE} \cite{Chev} yield the remaining entries. 
In particular, it follows easily from the character table that any two irreducible characters of the same degree have the same reduction modulo~$2$. 
\end{proof}

\begin{Corollary}\label{cor:simplesPSU(q)}%
The $B_{0}(k\mathsf{G})$-simple modules $\varphi$ and $\theta$ are not trivial source modules.
\end{Corollary}
\begin{proof}
It follows from the decomposition matrix of $B_{0}(k\mathsf{G})$ in Lemma~\ref{lem:2decPSU3} that $\varphi$ and $\theta$ are liftable modules.
Moreover, any lift of $\varphi$ to an $\mathcal{O}\mathsf{G}$-lattice affords the unipotent character $\chi^{}_{q(q-1)}$, and any lift of $\theta$ to an $\mathcal{O}\mathsf{G}$-lattice affords one of the characters $\chi_{(q-1)(q^2-q+1)}^{(u,v)}$ of degree $(q-1)(q^2-q+1)$. However, it follows from~\cite[II Theorem 12.4(iii)]{Lan83} that neither $\chi^{}_{q(q-1)}$ nor the characters  $\chi_{(q-1)(q^2-q+1)}^{(u,v)}$ can be the characters of trivial source modules, because it is easily checked from the character table that these characters take strictly negative values at some $2$-elements. (See e.g.~\cite[Table~3.1]{Gec90}.)
\end{proof}

Next we collect useful information about the permutation module  $k_{\mathbb B}{\uparrow}^{\mathsf{G}}$ and the 2nd Heller translate $\Omega^{2}(k_{\mathsf{G}})$,
based on ideas of  \cite[pp.\,259--260 and p.\,263]{OW02} and which complements the information provided in \cite[pp.\,227--228]{His04}.

\begin{Lemma}\label{OkyWaki}%
Assume $\mathsf{G}=\PGU_3(q)$ and set $X:=\Omega^2(k_{\mathsf{G}})$. Then, the following assertions hold:
\begin{enumerate} 
\renewcommand{\labelenumi}{\rm{(\alph{enumi})}} 
\item 
the permutation module $k_{\mathbb B}{\uparrow}^\mathsf{G}$ is a trivial source module  affording the ordinary character ${1_{\mathbb B}}{\uparrow}^\mathsf{G} =1_{\mathsf{G}}+\chi_{q^3}$
and satisfying 
$$ 
{k_{\mathbb B}}{\uparrow}^\mathsf{G} \ = \ 
\boxed{\begin{matrix}
k_\mathsf{G}\\ \varphi\\ \theta\\ \varphi\\k_\mathsf{G}
\end{matrix}}
\ = \ {\mathrm{Sc}}(\mathsf{G},\mathbb B)\ = \ {\mathrm{Sc}}(\mathsf{G}, Q)\,
$$
where $Q\in\Syl_2(\mathbb B)$ is such that  $Q\cong C_{2^{n+1}}$ and we may assume that $Q\leq P$;
\item
no indecomposable 
direct summand 
$U$ of $\varphi{\downarrow}_{\mathbb B}$ or  $\theta{\downarrow}_{\mathbb B}$ 
 belongs to $B_0(k\mathbb B)$;
\item 
$\Ext_{k\mathsf{G}}^1(k_\mathsf{G}, k_\mathsf{G})=0$;
\item 
$\dim_k\,\Ext_{k\mathsf{G}}^1(k_\mathsf{G}, \varphi)=\dim_k\,\Ext_{k\mathsf{G}}^1(\varphi,k_\mathsf{G})=1$;
\item 
$\Ext_{k\mathsf{G}}^1(k_\mathsf{G}, \theta)=\Ext_{k\mathsf{G}}^1(\theta,k_\mathsf{G})=0$;
\item  
$\hd(\Omega(k_\mathsf{G}))=\varphi$ and so there exists a surjective $k\mathsf{G}$-homomorphism 
$P(\varphi)\twoheadrightarrow\Omega(k_\mathsf{G})$;
\item
$X$ lifts to an $\mathcal O\mathsf{G}$-lattice which affords the character 
$\chi_{q(q-1)}+\chi_{q^3}$ and  as composition factor 
$X= k_\mathsf{G}+3\times\varphi+\theta$\,;
\item 
$\soc(X)\cong\varphi$ and $\varphi\mid \hd(X)$\,;
\item
$\dim_k\,\Hom_{k\mathsf{G}}(X, k_{\mathbb B}{\uparrow}^\mathsf{G})=
\dim_k\,\Hom_{k\mathsf{G}}(k_{\mathbb B}{\uparrow}^\mathsf{G},X)=1$;
\item
$k_\mathsf{G} \not| \,\, \soc^2(X)$;
\item

$X$ has a uniserial $k\mathsf{G}$-submodule $Z\cong k_{\mathbb B}{\uparrow}^\mathsf{G} /\soc(k_{\mathbb B}{\uparrow}^\mathsf{G})$
of the form 
$$
\boxed{
\begin{matrix}
k_\mathsf{G}\\ \varphi\\ \theta\\ \varphi
\end{matrix}
}
$$ 
and hence if $Y:=\rad(Z)=\boxed{
\begin{matrix}
\varphi\\ \theta\\ \varphi
\end{matrix}
}$  then 
$X/Y$ is of the form 
$\boxed{ \begin{matrix} \varphi \\ k_\mathsf{G} \end{matrix}}$ or of the form $k_\mathsf{G}\oplus\varphi$\,.

\end{enumerate}
\end{Lemma}

\begin{proof} 
(a) The claim about the structure of $Q$ is clear from the structure of $\mathbb{B}$. 
Hence, it is clear that $\mathrm{Sc}(\mathsf{G},\mathbb B) = \mathrm{Sc}(\mathsf{G}, Q)$  
(see e.g.~\cite[Corollary 4.8.5]{NT89}).
The claim about $k_{\mathbb B}{\uparrow}^\mathsf{G}$ being uniserial with the given composition series 
and the given ordinary character is given by  \cite[Theorem 4.1(c) and Appendix (pp. 238--241)]{His04}.
Then, as $k_{\mathbb B}{\uparrow}^\mathsf{G}$ is indecomposable, and 
${\mathrm{Sc}}(\mathsf{G},\mathbb B)$ 
is an indecomposable direct summand of $k_{\mathbb B}{\uparrow}^\mathsf{G}$ by definition, certainly 
\smallskip $k_{\mathbb B}{\uparrow}^\mathsf{G}={\mathrm{Sc}}(\mathsf{G},\mathbb B)$. 
\par
(b) Suppose that $U\,|\,\varphi{\downarrow}_{\mathbb B}$ and $U$ lies in $B_0(k\mathbb B)$.
Since $\mathbb B$ is $2$-nilpotent, its principal block is nilpotent and so 
$\Irr_{k}(B_0(k\mathbb B))=\{k_{\mathbb{B}}\}$.
All the composition factors of $U$ are isomorphic to $k_{\mathbb{B}}$ as they must lie in $B_0(k\mathbb B)$. 
Thus, $0\neq \Hom_{k\mathbb B}(U,k_{\mathbb B})$ and Frobenius reciprocity yields
$$0\neq \Hom_{k\mathbb B}(\varphi{\downarrow}_{\mathbb B}, k_{\mathbb B})
\cong \Hom_{kG}(\varphi, k_{\mathbb B}{\uparrow}^\mathsf{G})\,,$$
proving that $\varphi$ is a constituent of the socle of  $k_{\mathbb B}{\uparrow}^G$. 
This contradicts (a) and so the first claim follows. 
The claim about $\theta$ is proved \smallskip analogously. 
\par
(c) By \cite[I Corollary 10.13]{Lan83}, $\Ext_{k\mathsf{G}}^1(k_\mathsf{G}, k_\mathsf{G})=0$ 
\smallskip as~$O^2(\mathsf{G})=\mathsf{G}$.
\par
(d) First, it is immediate from (a) that $\dim_k\,\Ext_{k\mathsf{G}}^1(k_\mathsf{G},\varphi)\geq 1$. 
Now, suppose that $\dim_k\,\Ext_{k\mathsf{G}}^1(k_\mathsf{G},\varphi)\geq 2$. 
Then, 
there exists a non-split short exact sequence
$$0\rightarrow\varphi\rightarrow V\rightarrow k_{\mathbb B}{\uparrow}^\mathsf{G}\rightarrow 0$$
of $k\mathsf{G}$-modules, i.e. $\Ext_{k\mathsf{G}}^1(k_{\mathbb B}{\uparrow}^\mathsf{G},\varphi)\neq 0$. 
However, by the Eckmann--Shapiro Lemma, 
$$\Ext_{k\mathsf{G}}^1(k_{\mathbb B}{\uparrow}^\mathsf{G},\varphi)\cong
 \,\Ext_{k\mathbb B}^1(k_{\mathbb B},\varphi{\downarrow}_{\mathbb B})\,,$$ 
 which is zero by (b).  This is a contradiction and so it follows that 
 $\dim_k\,\Ext_{k\mathsf{G}}^1(k_\mathsf{G},\varphi)=1$. 
Moreover, $\dim_k\,\Ext_{k\mathsf{G}}^1(k_\mathsf{G},\varphi)=1$ 
as well  by the self-duality \smallskip of $k_{\mathsf{G}}$ and $\varphi$.
\par
(e) 
Suppose that $\Ext_{k\mathsf{G}}^1(k_\mathsf{G},\theta)\,{\not=}\,0$. 
Then, with arguments similar to those used  in the proof of (d), we obtain that 
$\Ext_{k\mathsf{G}}^1(k_{\mathbb B}{\uparrow}^\mathsf{G},\theta)\,{\not=}\,0$, 
which contradicts (b). 
Again, as $k_{\mathsf{G}}$ and~$\theta$ are self-dual, it follows that 
$\Ext_{k\mathsf{G}}^1(k_\mathsf{G},\theta)=0$ \smallskip as well. 
\par
(f)
Since $\Irr_{k}(B_{0}(k\mathsf{G}))=\{k_\mathsf{G}, \varphi, \theta\}$, 
it follows from (c), (d) and (e) that the second Loewy layer of $P(k_{\mathsf{G}})$ 
consists just of the simple module $\varphi$, with multiplicity $1$. 
Thus, the claim follows from the fact that \smallskip 
$\Omega(k_\mathsf{G})=P(k_\mathsf{G}){\cdot}{\rad}(kG)$.  
\par
(g)
First, it is well-known that  $X$ lifts to an $\mathcal O\mathsf{G}$-lattice 
(see e.g. \cite[\S7.3]{Las23}).
Moreover, by (f) we have that  $\Omega^2(k_\mathsf{G})$ is the kernel of a 
short exact sequence of $k\mathsf{G}$-modules of the form
$$
0 \rightarrow\Omega^2(k_\mathsf{G}) \rightarrow P(\varphi)\rightarrow\Omega(k_\mathsf{G})\rightarrow 0\,.
$$
Thus, in the Grothendieck ring of $kG$, we have
$$
\Omega^2(k_\mathsf{G}) = P(\varphi)-\Omega(k_\mathsf{G}) = P(\varphi)-P(k_\mathsf{G}){\cdot}\rad(k\mathsf{G}) 
= P(\varphi)-( P(k_\mathsf{G})-k_\mathsf{G})\,.
$$
Using the decomposition matrix of $B_{0}(k\mathsf{G})$ given in Lemma~\ref{lem:2decPSU3}, 
we obtain that the character afforded by $\Omega^2(k_\mathsf{G})$ is  $\chi_{q(q-1)}+\chi_{q^3}$,  
 and the composition factors of $X$ \smallskip are as~claimed.
 \par
(h)
It is clear that $\soc(X)\cong\varphi$ as $\Omega^{2}(k_\mathsf{G})$ is a submodule 
of $P(\varphi)$ by  the proof of assertion (g).
Now, by Lemma \ref{lem:2decPSU3}, any lift of $\varphi$ affords the 
character $\chi_{q(q-1)}$. Thus, 
by \cite[I Theorem 17.3]{Lan83}, $X$ has a pure submodule $Y$ affording the 
Steinberg character~$\chi_{q^{3}}$.
Then, $X/Y\cong \varphi$, proving \smallskip the claim.
\par
(i) It follows from Frobenius reciprocity that
$$
\Hom_{k\mathsf{G}}(k_{\mathbb B}{\uparrow}^\mathsf{G}, X) 
\cong \Hom_{k\mathbb B}(k_{\mathbb B}, X{\downarrow}_{\mathbb B}).
$$
Now, as  $\Irr_{k}(B_0(k\mathsf{G}))=\{k_{\mathsf{G}},\varphi,\theta\}$ and by 
assertion (g) we have that $k_{\mathsf{G}}$ 
has multiplicity one as a composition factor of $X$, it follows from (b) that 
$$
 \Hom_{k\mathbb B}(k_{\mathbb B}, X{\downarrow}_{\mathbb B})
\cong \Hom_{k\mathbb B}(k_{\mathbb B}, k_\mathsf{G}{\downarrow}_{\mathbb B}) \cong k\ \
$$
as $k$-vector spaces.  The second equality is obtained \smallskip analogously.  
\par
(j) 
Consider the Auslander--Reiten sequence $(\mathcal E):0\rightarrow X\overset{g}{\rightarrow} E
\overset{\pi}\rightarrow k_{\mathsf{G}}\rightarrow 0$
starting at $X=\Omega^{2}(k_{\mathsf{G}})$ (and hence ending at $k_{\mathsf{G}}$).
(See e.g. \cite[\S 34]{The95} for this notion.)
By (d) there exists a uniserial module of length $2$ of the form
$$\boxed{\begin{matrix}k_{\mathsf{G}}\\ \varphi\end{matrix}}=:Y\,.$$
Consider the quotient homomorphism 
$\rho:Y\rightarrow Y/\varphi\cong k_{\mathsf{G}}$, which is obviously not a split-epi. Hence
there exists a  $k{\mathsf{G}}$-homomorphism $\alpha: Y\rightarrow E$ with 
$\pi\circ\alpha =\rho$.
Next we claim that $\ker(\alpha) \not=\soc(Y)$.  So assume $\ker(\alpha) =\soc(Y)$. Then, 
$E\geq\im(\alpha)\cong k_{\mathsf{G}}$, proving that $\im(\alpha)\leq\soc(E)$ (as it is simple). 
On the other hand, by (h), $\soc(X)=\varphi$, implying that $\im(\alpha) \cap \soc(X)=0$.
Thus, identifying $X$ with its image in $E$, we get  that  $\im(\alpha)\cap X=0$ 
as  $\im(\alpha)$ is simple. (Use here the same argument as in the 
last five lines of the proof of Lemma~\ref{lem:simpleSocle}.) 
Hence, $E$ has a submodule of the form $\im(\alpha)\oplus X$, which implies that
$E=\im(\alpha)\oplus X$ as we can read from the s.e.s. $(\mathcal E)$
 that they have the same $k$-dimension. Thus, it follows from \cite[Lemma ~6.12]{Car96} that
the sequence $(\mathcal E)$ splits, which is a contradiction and the claim follows.
Next, since $\alpha \neq 0$, it follows that $\ker(\alpha)=0$, that is,  $\alpha$ is injective. 
Hence, $\im(\alpha)\cong Y$. 
Now, suppose that $k_{\mathsf{G}}\,|\,\soc^2(X)$. Set $W:=\soc^2(X)+\im(\alpha)\leq E$.
Note that $\im(\alpha) \not\leq X$ since $X=\ker(\pi)$, so that $\im(\alpha)\not\leq\soc^2(X)$.
Hence $\soc^2(X)+\im(\alpha)$ has 
the following socle series 
$$\begin{bmatrix}k_\mathsf{G} \ k_\mathsf{G} \\ \varphi\end{bmatrix}\,,$$
since by Lemma~\ref{lem:simpleSocle} we have $\soc(E)=\soc(X)\cong\varphi$, 
where the last isomorphism holds  by~(h). This is a contradiction to (d), 
and so the \smallskip claim follows. 
\par
(k)
It follows from assertions (i) and (a) that 
$$
1 = \dim_k\,\Hom_{k\mathsf{G}}(k_{\mathbb B}{\uparrow}^\mathsf{G},X) 
=
\dim_k\,\Hom_{k\mathsf{G}}\Biggl(\, \boxed{\begin{matrix}
k_\mathsf{G}\\ \varphi\\ \theta\\ \varphi\\k_\mathsf{G}
\end{matrix}},\,X\Biggr) 
=\dim_k\,\Hom_{k\mathsf{G}}\Biggl( \,
\boxed{\begin{matrix}
k_\mathsf{G}\\ \varphi\\ \theta\\ \varphi
\end{matrix}},\,X\Biggr),
$$
where $\boxed{\begin{matrix}
k_\mathsf{G}\\ \varphi\\ \theta\\ \varphi
\end{matrix}}:= k_{\mathbb B}{\uparrow}^\mathsf{G} /\soc(k_{\mathbb B}{\uparrow}^\mathsf{G})$ and the
 last equality holds because $c_X(k_\mathsf{G})=1$ by (g).  Therefore, there exists a non-zero 
 $k\mathsf{G}$-homomorphism 
 $$\gamma: \boxed{\begin{matrix}
k_\mathsf{G}\\ \varphi\\ \theta\\ \varphi
\end{matrix}} \longrightarrow X
 $$ 
Now,  either $\gamma$ is injective and we are done, or $\ker(\gamma)\neq 0$. 
In the latter case, by (h) we have  $\ker(\gamma)= \boxed{\begin{matrix}
 \theta\\ \varphi
\end{matrix}}$
and so there is an injective homomorphism
$$j:\, \boxed{\begin{matrix}
 k_{\mathsf{G}}\\ \varphi
\end{matrix}}\hookrightarrow X\,.$$
which contradicts assertion (j). The claim follows. 
\end{proof}

We can now prove the main result of this section.\\

\begin{Proposition}\label{lem:FinalStroke}\label{prop:PSU3}%
For each $i\in\{1,2\}$ let  $\mathsf{G}_{i}:=\mathsf{G}(q_{i})=\PGU_{3}(q_{i})$ and 
$G_{i}:=G(q_{i})=\PSU_3(q_i)$, where we assume that $(q_{i}+1)_2=2^n$. 
Then,  the following assertions hold:
\begin{enumerate}
\item[\rm(a)] $\mathrm{Sc}(\mathsf{G}_1\times \mathsf{G}_2,\Delta P)$ induces a splendid 
Morita equivalence between $B_0(k\mathsf{G}_1)$ and $B_0(k\mathsf{G}_2)$;
\item[\rm(b)]  $\mathrm{Sc}(G_1\times G_2,\Delta P)$ induces a splendid Morita
equivalence between $B_0(kG_1)$ and $B_0(kG_2)$. 
\end{enumerate}
\end{Proposition}

\begin{proof}
Again, as $G_{i}\lhd\mathsf{G}_{i}$ and $|\mathsf G_i/G_i|=(3,q+1)$ is odd for each 
$i~\in~\{1,2\}$, by Lemma~\ref{lem:CGp5956} and Lemma~\ref{lem:Aut(P)}, 
assertion (b) follows from assertion (a), so it suffices to prove (a).
\par
Set $M:={\mathrm{Sc}}(\mathsf{G}_1\times \mathsf{G}_2,\Delta P)$. For each $i~\in~\{1,2\}$ 
write  $B_{i}:=B_{0}(k\mathsf{G}_{i})$. Write $\Irr_{k}(B_{i})=\{k_{\mathsf{G}_{i}},\varphi_{i},\theta_{i}\}$
with $\dim_k\,\varphi_i={q_i}(q_i-1)$ and $\dim_k\,\theta_i=(q_i-1)({q_i}^2-q_i+1)$, 
and set $\mathbb{B}_{i}:=\mathbb{B}_{i}(q_{i})$
as in Notation~\ref{nota:PGU}. Moreover, let $Q_i\in\Syl_{2}(\mathbb B_i)$ such that 
$Q_{i}\leq P$ and let $X_{i}:=\Omega^{2}(k_{\mathsf{G}_{i}})$ as in Lemma~\ref{OkyWaki}.
Furthermore,  observe that 
$\mathcal F_P(\mathsf{G}_1)=\mathcal F_P(\mathsf{G}_2)$  and 
all involutions in $\mathsf{G}_i$ are $\mathsf{G}_i$-conjugate (see e.g. \cite[Theorem~5.3]{CG12} 
and/or \cite[Proposition~2 on p.11]{ABG70}).  It follows that it suffices to prove that Conditions~
(I) and~(II) of Theorem~\ref{Thm:ProveExSMEQ} \smallskip hold.\\
{\bf Condition (I)}.
A similar argument to the one used  in the proof of Proposition~\ref{prop:PSL3Puig} 
(Condition I) can be used. In the present case, if $z$ is an involution in the centre of $P$, 
then $C_{\mathsf{G}_{i}}(z)=:C_i$ is a quotient of $\GU_{2}(q)$ by a normal subgroup of odd 
index by \cite[Proposition 4(iii)]{ABG70}.
Hence, we obtain from Lemma~\ref{lem:G/Op'(G)} and Proposition~\ref{pro:SU2^m}
that $M_{z}:=\Sc(C_1\times C_2,\Delta P)$ induces a splendid Morita equivalence between 
$B_0(kC_1)$ and $B_0(kC_2)$ by Proposition~\ref{pro:SL2^m}(a). Moreover, 
$M_z=M(\Delta\langle z\rangle)$  by the Brauer indecomposability of $M$ proved 
in~\cite[Theorem~1.1]{KT20}, proving that Condition (I) is \smallskip verified.\\
{\bf Condition (II)}. 
Again, we have to prove that the functor $-\otimes_{B_1}M$ maps 
the simple $B_{1}$-modules to the simple $B_{2}$-modules, and again, we have 
$k_{\mathsf{G}_1}\otimes_{B_1}M\cong k_{\mathsf{G}_2}$ by \cite[Lemma~3.4(a)]{KL20a}. 
Thus, it remains to prove that $\varphi_{1}\otimes_{B_{1}}M\cong \varphi_{2}$ 
and $\theta_{1}\otimes_{B_{1}}M\cong \theta_{2}$.
\par
First recall from Lemma~\ref{OkyWaki}(a) that for each $i\in\{1, 2\}$ we have 
\begin{equation}\label{ScottBorel}
\Sc({\mathsf{G}}_i, Q_i)
=\boxed{\begin{matrix}k_{{\mathsf{G}}_i}\\ \varphi_i\\ \theta_i\\ \varphi_i\\k_{{\mathsf{G}}_i}\end{matrix}}\,.
\end{equation} and moreover  by \cite[Lemma~3.4(c)]{KL20a} we have 
$$\Sc({\mathsf{G}}_1, Q_1)\otimes_{B_{1}}M\cong 
\Sc({\mathsf{G}}_2, Q_2)\oplus (\text{proj})\,.$$
Thus, because we already know that 
$$\soc(\Sc({\mathsf{G}}_1, Q_1))\otimes_{B_1}M=k_{\mathsf{G}_1}\otimes_{B_1}M
\cong k_{\mathsf{G}_2}=\soc(\Sc({\mathsf{G}}_2, Q_2))\,,$$
the stripping-off method  (see Lemma~\ref{StrippingOffMethod}(a))
yields 
\begin{equation}\label{stripping-off-ScotBorel/socle} 
\boxed{\begin{matrix}k_{{\mathsf{G}}_1}\\ \varphi_1\\ \theta_1\\ \varphi_1\end{matrix}}
\otimes_{B_1}M \ \cong \ 
\boxed{\begin{matrix}k_{{\mathsf{G}}_2}\\ \varphi_2\\ \theta_2\\ \varphi_2\end{matrix}}
\oplus\text{(proj)}
\end{equation}
where for each $i\in\{1,2\}$, the latter uniserial module of length 4 is defined to be 
$$\boxed{\begin{matrix}k_{{\mathsf{G}}_i}\\ \varphi_i\\ \theta_i\\ \varphi_i\end{matrix}}
:=\Sc({\mathsf{G}}_i, Q_i)/\soc(\Sc({\mathsf{G}}_i, Q_i))=:Z_{i}$$ 
as in Lemma~\ref{OkyWaki}(k). Then, applying again the stripping-off method 
(Lemma~\ref{StrippingOffMethod}(b) this time) to equation (\ref{stripping-off-ScotBorel/socle}) and $\hd{Z}_{i}\cong k_{\mathsf{G}_{i}}$ ($i\in\{1,2\}$), we obtain that
\begin{equation}\label{Y_i}
\boxed{\begin{matrix}\varphi_1\\ \theta_1\\ \varphi_1\end{matrix}}
\otimes_{B_1} M
\ = \ 
\boxed{\begin{matrix}\varphi_2\\ \theta_2\\ \varphi_2\end{matrix}}
\oplus\text{(proj)}
\end{equation}
where for each $i\in\{1,2\}$, the latter uniserial module of length 3 is defined to be 
$$\boxed{\begin{matrix}\varphi_i\\ \theta_i\\ \varphi_i\end{matrix}}:=\rad(Z_{i})=:Y_{i}\,,$$
again as in Lemma~\ref{OkyWaki}(k).
Now, by the proof of Lemma~\ref{OkyWaki}(k), we also know that $Y_{i}$ is (up to identification) 
a submodule of $X_{i}$ for each $i\in\{1,2\}$, and 
$$X_{1}\otimes_{B_{1}}M\cong X_{2}\oplus (\text{proj})$$
by \cite[Lemma~3.4(d)]{KL20a}. Because of the way, we have defined $X_{i}$ and $Y_{i}$
 ($i\in\{1,2\}$) via the 
stripping-off method, it follows from the exactness of the functor $-\otimes_{B_{1}}M$ that 
$$X_{1}/Y_{1}\otimes_{B_{1}}M\cong  (X_{1}\otimes_{B_{1}}M)/(Y_{1}\otimes_{B_{1}}M)\cong X_{2}/Y_{2}\oplus (\text{proj})\,.$$ 
Lemma~\ref{OkyWaki}(k) gives, up to isomorphism, two possibilities for $X_{1}/Y_{1}$ and  two possibilities 
for $X_{2}/Y_{2}$, namely,
$$ \boxed{ \begin{matrix} \varphi_{1} \\ k_{\mathsf{G}_{1}}\end{matrix}}\text{ or }k_{\mathsf{G}_{1}}\oplus\varphi_{1}, \text{ and } \boxed{ \begin{matrix} \varphi_{2} \\ k_{\mathsf{G}_{2}}\end{matrix}}\text{ or }k_{\mathsf{G}_{2}}\oplus\varphi_{2}    \text{, respectively},$$
but in any configuration we can apply 
the stripping-off method again (Lemma~\ref{StrippingOffMethod}(a)) to strip off the trivial 
socle summand  of $X_{1}/Y_{1}$ and $X_{2}/Y_{2}$ and we obtain that
$$\varphi_{1}\otimes_{B_{1}}M\cong \varphi_{2}\oplus\text{(proj)}\,.$$
However, as $\varphi_{1}$ is simple, $\varphi_{1}\otimes_{B_{1}}M$ must be indecomposable 
by \cite[Theorem~2.1(a)]{KL20a}, 
proving that $\varphi_{1}\otimes_{B_{1}}M\cong \varphi_{2}$. Then, we can apply yet again the 
stripping-off method twice (once Lemma~\ref{StrippingOffMethod}(a) and once 
Lemma~\ref{StrippingOffMethod}(b)) to equation (\ref{Y_i}) and $\soc(Y_{i})$, 
respectively $\hd(Y_{i})$, ($i\in\{1,2\}$) to obtain that
$$\theta_{1}\otimes_{B_{1}}M\cong \theta_{2}\oplus\text{(proj)}\,.$$
However, again, as $\theta_{1}$ is simple, $\theta_{1}\otimes_{B_{1}}M$ must be indecomposable 
by \cite[Theorem~2.1(a)]{KL20a}, eventually proving that $\theta_{1}\otimes_{B_{1}}M\cong \theta_{2}$.
\end{proof}

\section{Proofs of Theorem~\ref{thm:MainResult} and Theorem~\ref{thm:MoritaClasses}.}\label{sec:proofMainb}

We can now prove our main results, that is, Theorem~\ref{thm:MainResult} and Theorem~\ref{thm:MoritaClasses}. 
We recall that $G$ is a finite group with a fixed  Sylow $2$-subgroup $P\cong C_{2^n}\wr C_2$, where $n\geq 2$ is a fixed integer.

\begin{proof}[Proof of Theorem~\ref{thm:MainResult}]
{\rm (a)}
To start with, by Lemma~\ref{lem:G/Op'(G)}, we may assume that ${O_{2'}(G)=1}$ 
and therefore that $G$ is one of the groups listed in~Theorem~\ref{thm:classification_wr}. 
Furthermore, by Lemma~\ref{lem:CGp5956} and Lemma~\ref{lem:Aut(P)},
we may also assume  that $O^{2'}(G)=G$. Hence, 
 Theorem~\ref{thm:classification_wr}, applied a second time, implies 
 that $G$ belongs to family {\sf(Wj(n))} for some ${\sf j\in\{1,\cdots,6\}}$.
 \par
 It remains to prove that
{\sf j} is uniquely determined. So, suppose that $G=:G_{1}$ is a finite group belonging to 
family  ${\sf (Wj_1(n))}$ for  some ${\sf j_1}\in\{{\sf1},\cdots,{\sf6}\}$ and assume that  the 
following hypothesis is satisfied:
\begin{itemize}
\item[($\ast$)] $B_{0}(kG_{1})$ is splendidly Morita equivalent to the principal block 
$B_{0}(kG_{2})$ of a finite group $G_{2}$ belonging to family  ${\sf (Wj_2(n))}$  for  some 
${\sf j_2}\in\{{\sf 1},\cdots,{\sf 6}\}$.
\end{itemize}
For $i\in\{1,2\}$ set $B_i:=B_0(kG_i)$, and notice that ($\ast$) implies that 
$\ell(B_{1})=\ell(B_{2})$ and $k(B_{1})=k(B_{2})$ because these numbers 
are invariant under Morita equivalences. 
\par
Now, first assume that ${\sf j_1} = 1$. Then, it follows from Theorem~\ref{thm:k(B)l(B)} 
that $\ell(B_{1})=1$
and  $\ell(B_{2})>1$ if ${\sf j_2} > 1$, contradicting ($\ast$). Hence, we 
have ${\sf j_2} = 1$ and $G_{2}\cong G_{1}$. 
\par
Assume then  that  ${\sf j_1} = 2$. Then, by Theorem~\ref{thm:k(B)l(B)}, we have 
$\ell(B_{1})=2$ and by ($\ast$) we may also assume that ${\sf j_2}\in\{{\sf 2,3,4}\}$. 
If ${\sf j_2} \neq {\sf2}$, then, as $n\geq 2$, Theorem~\ref{thm:k(B)l(B)} yields 
\[
k(B_{1})= (2^{2n-1}+9{\cdot}2^{n-1}+4)/3\neq2^{2n-1}+2^{n+1}=k(B_{2})\,,
\] 
also contradicting ($\ast$), so that  ${\sf j_2} = 2$ and $G_{2}\cong G_{1}$. 
\par
Suppose next that $\sf j_1 = 3$.  Then, again by Theorem~\ref{thm:k(B)l(B)}, we have 
$\ell(B_{1})=2$ and by ($\ast$) we may assume that ${\sf j_2}\in\{{\sf 2,3,4}\}$. 
Moreover, by the previous case, we have ${\sf j_2}\neq {\sf 2}$. So, let us 
assume that ${\sf j_2}={\sf 4}$.
We can consider that $B_{1} =B_0(k\SL_2^n(q_1))$ and $B_{2}=B_0(k\SU_2^n(q_2))$ 
for prime powers $q_1$, $q_{2}$ such that  $({q_1-1)_2=2^n=(q_2+1)_2}$.  
Then, again, as $\SL_2^n(q_1)\lhd \GL_2(q_1)$ and $\SU_2^n(q_2)\lhd \GU_2(q_2)$ 
are normal subgroups of odd index, it follows from Lemma~\ref{lem:CGp5956} and Lemma~\ref{lem:Aut(P)}
that $B_0(k\GL_2(q_1))$ and $B_0(k\GU_2(q_2))$ are splendidly Morita 
equivalent,
and so 
Lemma~\ref{lem:G/Z(G)} implies that 
$B_0(k\PGL_2(q_1))$ and $B_0(k\PGU_2(q_2))$ are splendidly Morita equivalent.
Now, as $\PGL_2(q_1)\cong\PGU_2(q_1)$,  
we have that 
$\mathcal B_1:=B_0(k\PGL_2(q_1))$ and $\mathcal B_2:=B_0(k\PGL_2(q_2))$ are
splendidly Morita equivalent, 
where 
$D_{2^{n+1}}\in\Syl_2(\PGL_2(q_1))\cap\Syl_2(\PGL_2(q_2))$
by the proofs of Proposition~\ref{pro:SL2^m} and Proposition~\ref{pro:SU2^m}.
However, the conditions on $q_1$ and $q_2$ imply that 
$\mathcal B_1$ and $\mathcal B_2$  are in (5) and (6), respectively,
in the list of \cite[Theorem 1.1]{KL20a},  hence cannot be splendidly Morita equivalent.
Thus, we have a contradiction, proving  that ${\sf j_2}={\sf 3}$ if ${\sf j_1}={\sf 3}$. 
Moreover, swapping the roles of ${\sf j_1}$ and ${\sf j_2}$ in the previous argument, we 
obtain that ${\sf j_2}={\sf 4}$ if ${\sf j_1}={\sf 4}$. 
\par
Suppose next that ${\sf j_1}={\sf 5}$. Then, as above, Theorem~\ref{thm:k(B)l(B)} and ($\ast$) imply that $\ell(B_{1})=3$
 and ${\sf j_2}\in\{{\sf5},{\sf6}\}$.  So, assume  that ${\sf j_2}={\sf 6}$.
Hence, we can consider that $B_{1}=B_0(k\PSL_3(q_1))$ and $B_{2}=B_0(k\PSU_3(q_2))$ 
for prime powers $q_1$ and $q_2$ such that $(q_1-1)_2=2^n=(q_2+1)_2$.
However, $B_0(k\PSL_3(q_1))$ and $B_0(k\PSU_3(q_2))$ cannot be splendidly Morita equivalent 
by Lemma~\ref{lem:tsPSL_{3}(q)} and Corollary~\ref{cor:simplesPSU(q)}, because such an equivalence 
maps simple modules to simple modules and also trivial source modules to trivial source modules. 
 It follows that ${\sf j_2}={\sf 5}$ if ${\sf j_1}={\sf 5}$.
Again, swapping the roles of ${\sf j_1}$ and ${\sf j_2}$ in the previous argument, we 
obtain that ${\sf j_2}={\sf 6}$ if ${\sf j_1}={\sf 6}$. 
Finally, we observe that the claim about the Scott module is immediate by construction.\\
{\rm (b)}
Assume $G_1$ and $G_2$ both belong to family {\sf (Wj(n))} for a ${\sf j}\in\{\sf 3,{\sf4},{\sf5},{\sf6}\}$. 
Then, $\mathrm{Sc}(G_{1}\times G_{2}, \Delta P)$ induces a splendid Morita equivalence between  
$B_0(kG_{1})$ and $B_0(kG_{2})$ by Propositions~\ref{pro:SL2^m}, \ref{pro:SU2^m}, \ref{prop:PSL3Puig} and 
\ref{prop:PSU3} for $\sf j=\sf 3,4,5$ and~$\sf6$ respectively.
\end{proof}

\begin{proof}[Proof of Theorem~\ref{thm:MoritaClasses}]
It is clear from the definitions that any splendid Morita equivalence is in particular a Morita equivalence. Thus, it only remains to prove that two distinct splendid Morita equivalence classes of principal blocks in Theorem~\ref{thm:MainResult} do not merge into one Morita equivalence class.  In fact,  from the numbers $\ell(B)$ and $k(B)$ in Theorem~\ref{thm:k(B)l(B)}, it suffices to argue that the splendid Morita equivalence classes of principal blocks of groups of type {\sf (Wj(3))} and {\sf (Wj(4))}, respectively of type {\sf (Wj(5))} and {\sf (Wj(6))}, do not merge into one Morita equivalence class. 
In the former case, this is clear from the proof of Theorem~\ref{thm:MainResult}, because else $B_0(k\PGL_2(q_1))$ and $B_0(k\PGL_2(q_2))$  with $({q_1-1)_2=2^n=(q_2+1)_2}$ would be Morita equivalent, which would contradict Erdmann's classification of tame blocks in~\cite{Erd90}.
In the latter case, it follows from the decomposition matrices of $B_0(k\PSL_3(q_1))$ and $B_0(k\PSU_3(q_2))$  with $({q_1-1)_2=2^n=(q_2+1)_2}$ given in \cite[Proposition~6.12]{Sch15} and Lemma~\ref{lem:2decPSU3}, respectively, that these blocks are not Morita equivalent. The claim follows. 
\end{proof}

\appendix 

\section{Appendix. On  \cite[Proposition 3.3(b)]{KL20b} }\label{App:A}%

The purpose of this appendix  is to fix a problem in the proof of \cite[Proposition~3.3(b)]{KL20b},
which was incomplete as written in \cite{KL20b}.  See Remark~\ref{Rem4.4}. \\

The next corollary is  the most essential in this appendix, namely, for the proof of 
(ii) $\Rightarrow$ (i) in Theorem \ref{thm:ScottFor_G-G/Z}.
This is already implicitly explained in the proof of \cite[Theorem 9.7.4]{Lin18}.

\begin{Corollary}[{}{See \cite[Theorem 9.7.4]{Lin18}}]\label{cor:Lin18-9.7.4}
The notation here is the same as that in \cite{Lin18} except that $\mathcal O$ is replaced by $k$.
Let $G, H$ be finite groups, let $b, c$ be blocks of $kG, kH$, respectively,
such that $kGb$ and $kHc$ have a common defect group $P$. Let $i$ and $j$ be
$P$-source idempotents of $kGb$ and $kHc$, respectively 
(and hence $i\in(kGb)^{\Delta P}$ and $j\in(kHc)^{\Delta P}$).
\par
Now, suppose that there is an indecomposable direct summand $M$ of the
$(kGb, kHc)$-bimodule $kGi\otimes_{kP}jkH$ such that the pair $(M, M^*)$ induces a Morita 
equivalence between $kGb$ and $kHc$.
Furthermore let $\varphi:=\varphi_M$ be the unitary interior $P$-algebra isomorphism 
${\varphi:i\,kG i\overset{\approx}{\rightarrow} j kHj}$ induced by $M$ as in
\cite[Theorem 9.7.4]{Lin18}.
Then, for any indecomposable right $kGb$-module $X$, 
\[Xi \ \cong \ (X\otimes_{kGb}\,Mj)_\varphi\text{ \ as right \,}i\,kG i\text{-modules}\]
where $(X\otimes_{kGb}Mj)_\varphi = X\otimes_{kGb}\,Mj$
as $k$-vector spaces and the right action of $i\,kGi$ is defined  using of $\varphi$.
\end{Corollary}

\begin{proof}
The notation here is the same as in the proof of \cite[(i) $\Rightarrow$ (ii) in Theorem~9.7.4]{Lin18}
except that $\mathcal O$ is replaced by $k$.
First we know already $Mj\cong kGi$ as $(kG, kP)$-bimodules via $\alpha$ there.
In the following the endomorphism ring of a left $R$-module $X$ for a ring $R$ is denoted by $\End_R(X)$.
Then, since $Mj$ can be considered as a right not only $kP$-module but also $\End_{kGb}(Mj)$-module.
Since $Mj\cong kGi$ as left $kGb$-modules from the above,
it follows that $\End_{kGb}(Mj)\cong\End_{kGb}(kGi)\cong (i\,kGi)^{\text{op}}$ as $k$-algebras.
On the other hand, 
$\End_{kGb}(Mj)\cong\End_{kGb}(M\otimes_{kHc}\,kHj)\cong\End_{kHc}(kHj)\cong (j\,kHj)^{\text{op}}$ as
$k$-algebras
(the second isomorphism comes from the fact that
$M$ realises a Morita equivalence between $kGb$ and $kHc$). 
Since the isomorphism $\varphi=\varphi_M$ is defined by using these isomorphisms
(see the final several lines in the proof of \cite[(i) $\Rightarrow$ (ii) in Theorem~9.7.4]{Lin18}),
we eventually obtain that 
$(X \otimes_{kGb}Mj)_\varphi \cong X\otimes_{kGb} kGi\cong Xi$ as right $i\,kGi$-modules.
\end{proof}

\begin{Theorem}[{}{See Proposition 3.3(b) in \cite{KL20b}}]\label{Prop3.3}\label{thm:ScottFor_G-G/Z}
Suppose that $G_1$ and $G_2$ are finite groups with a common Sylow $p$-subgroup $P$,
and assume that 
$Z$ is a subgroup of $P$ such that $Z\leq Z(G_1)\cap Z(G_2)$. Write $\overline{G_1}:=G_1/Z$,
$\overline{G_2}:=G_2/Z$ and $\overline P:=P/Z$.  
Then, the following assertions are equivalent:
   \begin{itemize}
\item[\rm(i)] ${\mathrm{Sc}}(G_1\times G_2,\Delta P)$ induces a Morita equivalence 
between $B_0(kG_1)$ and $B_0(kG_2)$;
\item[\rm(ii)] ${\mathrm{Sc}}(\overline{G_1}\times\overline{G_2}, \Delta \overline P)$  
induces a Morita equivalence between $B_0(k\overline{G_1})$ and $B_0(k\overline{G_2})$.
   \end{itemize}
\end{Theorem}
\begin{proof}
Let $i\in\{1,2\}$. 
Write $B_i:=B_0(kG_i)$ and let $\overline{B_i}$ be the image of $B_i$
under the {$k$-algebra} epimorphism $kG_i\twoheadrightarrow k\overline{G_i}$
induced by the quotient group homomorphism $G_i\twoheadrightarrow\overline{G_i}$.
Then \cite[Chap. 5 Theorem~8.11]{NT89} says that
$\overline{B_i}=B_0(k\overline{G_i})$.
Furthermore
$\overline{B_i}\cong k\overline{G_i}\otimes_{kG_i}B_i\otimes_{kG_i} {k\overline{G_i}}$
as $(k\overline{G_i},k\overline{G_i})$-bimodules.
Write $M:={\mathrm{Sc}}(G_1\times G_2,\Delta P)$ 
and {$N:=M^*={\mathrm{Sc}}(G_2\times G_1,\Delta P)$}. 
Set $\overline M:=k\overline{G_1}\otimes_{kG_1}M\otimes_{kG_2}k\overline{G_2}$.
Then, the following holds:\\
\begin{align*}
\overline M&\,\,\Big|\,\,
k\overline{G_1}\otimes_{kG_1}
\Big( {\mathrm{Ind}}_{\Delta P}^{G_1\times G_2}{(k_{\Delta P})}
\Big)
\otimes_{kG_2}\,k\overline{G_2}
\\&=
k\overline{G_1}\otimes_{kG_1}(kG_1\otimes_{kP}\,kG_2)\otimes_{kG_2}\,k\overline{G_2}
\\&\cong
k\overline{G_1}\otimes_{kP}\,k\overline{G_2}
\\&\cong
k\overline{G_1}\otimes_{k\overline P}\,k\overline{G_2}
\\&\cong 
{\mathrm{Ind}}_{\Delta\overline P}^{\overline{G_1}\times\overline{G_2}}(k_{\Delta\overline P}).
\end{align*}
Note furthermore that 
$\overline M$ obviously has the trivial $k(\overline{G_1}\times\overline{G_2})$-module
$k_{\overline{G_1}\times\overline{G_2}}$ as an epimorphic image.
Set $\mathfrak M:={\mathrm{Sc}}(\overline{G_1}\times\overline{G_2},\,\Delta{\overline P})$.
Then
\begin{equation}\label{Prop3.3(b)} 
\mathfrak M
\,\Big|\,\overline M \quad
\text{ (equality does not necessarily hold)}.
\end{equation}
(i) $\Rightarrow$ (ii):
Set $\overline N:=k\overline{G_2}\otimes_{kG_2}N\otimes_{kG_1}k\overline{G_1}$.
Then, 
\begin{align*}
\overline M\otimes_{\overline{B_2}}\overline N
&\cong
\overline M\otimes_{k\overline{G_2}}\overline N
\\&\cong
(k\overline{G_1}\otimes_{kG_1}M\otimes_{kG_2}k\overline{G_2})\otimes_{k\overline{G_2}}\,
(k\overline{G_2}\otimes_{kG_2}N\otimes_{kG_1}k\overline{G_1})
\\&\cong
k{\overline G_1}\otimes_{kG_1}(M\otimes_{kG_2}k\overline{G_2})\otimes_{kG_2}N\otimes_{kG_1}k\overline{G_1}
\\&\cong
k{\overline G_1}\otimes_{kG_1}(k\overline{G_1}\otimes_{kG_1}M)\otimes_{kG_2}N\otimes_{kG_1}k\overline{G_1}
\\& \qquad\qquad
\text{ since } M\otimes_{kG_2}k\overline{G_2}\cong k\overline{G_1}\otimes_{kG_1}M \ \ 
\text{ as }(kG_1,kG_2)\text{-bimodules}
\\&\cong
(k{\overline G_1}\otimes_{kG_1}k\overline{G_1})\otimes_{kG_1}M\otimes_{kG_2}N\otimes_{kG_1}k\overline{G_1}
\\&\cong
(k{\overline G_1}\otimes_{k\overline{G_1}}k\overline{G_1})
                   \otimes_{kG_1}M\otimes_{kG_2}N\otimes_{kG_1}k\overline{G_1}
\\&\cong
k\overline{G_1}\otimes_{kG_1}(M\otimes_{kG_2}N)\otimes_{kG_1}k\overline{G_1}
\\&\cong
k\overline{G_1}\otimes_{kG_1}B_1\otimes_{kG_1}k\overline{G_1}\ \ \text{ by (i)}
\\&\cong
\overline{B_1}.
\end{align*}
Since $\overline{B_i}$ is a symmetric $k$-algebra for $i=1, 2$, the above already shows that
the pair $(\overline M, {\overline N})$ induces a Morita equivalence between
$\overline{B_1}$ and $\overline{B_2}$,
and hence $\overline M$ is indecomposable as a right
$k(\overline G_1\times\overline G_2)$-module,
which implies that $\mathfrak M\cong\overline M$
from (\ref{Prop3.3(b)}).
\medskip

\noindent (ii) $\Rightarrow$ (i):
As in \cite[p.822]{Lin01}
there exist $P$-source idempotents $j_i$ of $B_i$ for $i=1,2$ with
$M\,|\, (kG_1\,j_1\otimes_{kP}\,j_2\,kG_2)$.
Then, for $i=1,2$,  the image $\overline{j_i}$ of $j_i$
via the canonical $k$-algebra epimorphism $kG_i\twoheadrightarrow k\overline{G_i}$ is a  
$\overline P$-source idempotent of $\overline{B_i}$ 
(see \cite[Chap.5 Theorem 8.11]{NT89}, \cite[\S 3]{KP90} and \cite[Lemma~4.1]{Kaw03}).
Hence,
\begin{equation}\label{ScottForOvelineG}
\mathfrak M\,|\, (k\overline{G_1}\,\overline{j_1}\otimes_{k\overline P}\,\overline{j_2}\,k\overline{G_2})
\end{equation}
from (\ref{Prop3.3(b)}).
Now, we apply \cite[(i) $\Rightarrow$ (ii) in Theorem~9.7.4]{Lin18} to the blocks $\overline{B_1}$ and
$\overline{B_2}$.
Namely, the existence of such an $\mathfrak M$ induces 
a unitary interior $\overline P$-algebra isomorphism 
$$\Phi:=\varphi_{\mathfrak M}\ : \  \overline{j_1}\,k\overline{G_1}\,\overline{j_1}\ 
\overset{\approx}{\rightarrow}\ \overline{j_2}\,k\overline{G_2}\,\overline{j_2}.$$
Thanks to \cite[Corollary 1.12]{Pui01} (see \cite{Kue01}), $\Phi$ lifts to a unitary interior
$P$-algebra isomorphism $\varphi\,:\, j_1 kG_1\,j_1 \overset{\approx}{\rightarrow} j_2 kG_2\,j_2$,
that is, $\Phi(\overline a)=\overline{\varphi(a)}$ for every $a\in j_1\,kG_1\,j_1$.
Then, by making use of $\varphi$ it follows from \cite[(ii) $\Rightarrow$ (i) in Theorem 9.7.4]{Lin18}
that there is an indecomposable direct summand $\mathcal M$ of the
$(B_1,B_2)$-bimodule $kG_1\,j_1\otimes_{kP}j_2\,kG_2$ such that the pair 
$(\mathcal M,{\mathcal M}^\vee)$
induces a Morita equivalence between $B_1$ and $B_2$
Set $S_2:=k_{G_1}\otimes_{B_1}\mathcal M$. 
Then, though we do not know that $\mathcal M=\Sc(G\times H,\Delta P)$ yet,
Corollary~\ref{cor:Lin18-9.7.4} yields that 
$(S_2\,j_2)_\varphi \cong k_{G_1}\,{j_1}$ as right $j_1\,B_1\,j_1$-modules.
Hence by sending these via the canonical epimorphism $kG_i\twoheadrightarrow k{\overline{G_i}}$,
\begin{equation}\label{CorA1-mathfrakM}
(\overline{S_2}\,\overline{j_2})_{\Phi} \cong
(\overline{S_2}\,\overline{j_2})_{\overline{\varphi}}
\cong k_{\overline{G_1}}\,\overline{j_1} \ 
\text{ as right }\overline{j_1}\,\overline{B_1}\,\overline{j_1}\text{-modules.}
\end{equation}

Now, we claim that $\mathcal M\cong M$.
Since the Scott module $\mathfrak M$ induces a Morita equivalence
between $\overline{B_1}$ and $\overline{B_2}$ and $\mathfrak M$ is related to $\overline{j_i}$
for $i=1, 2$ 
(see (\ref{ScottForOvelineG})) and $\Phi$, it follows from Corollary~\ref{cor:Lin18-9.7.4} and 
\cite[Lemma 3.4(a)]{KL20a} that
\begin{equation}\label{CorA1} 
k_{\overline{G_1}}\,\overline{j_1}
\ \cong \ (k_{\overline{G_1}}\otimes_{\overline{B_1}}\,\mathfrak M\,\overline{j_2})_\Phi
\ \cong \ (k_{\overline{G_2}}\,\overline{j_2})_\Phi   
\ \ \text{ as right }\overline{j_1}\overline{B_1}\,\overline{j_1}\text{-modules.}
\end{equation}
Hence (\ref{CorA1-mathfrakM}) and (\ref{CorA1}) yield that
$(\overline{S_2}\,\overline{j_2})_{\Phi} \cong
(k_{\overline{G_2}}\,\overline{j_2})_\Phi$ 
as right $\overline{j_1}\overline{B_1}\,\overline{j_1}$-modules, and hence
$\overline{S_2}\,\overline{j_2} \cong k_{\overline{G_2}}\,\overline{j_2}$ 
as right $\overline{j_2}\overline{B_2}\,\overline{j_2}$-modules.
Since the pair $(\overline{B_2}\,\overline{j_2}, \overline{j_2}\overline{B_2})$ induces
the canonical Morita equivalence between $\overline{B_2}$ and its source algebra
$\overline{j_2}\,\overline{B_2}\,\overline{j_2}$, we get that
$\overline{S_2}\cong k_{\overline{G_2}}$ as right $\overline{B_2}$-modules,
say as right $k\overline{G_2}$-modules.
Since $S_2$ and $k_{G_2}$ are both simple right $kG_2$-modules, 
$Z$ is in their kernels. Thus,
${\mathrm{Def}}^{G_2}_{\overline{G_2}}\, S_2 \cong {\mathrm{Def}}^{G_2}_{\overline{G_2}}\,k_{G_2}$
as right $k\overline{G_2}$-modules where ${\mathrm{Def}}$ is the deflation.
Hence $S_2\cong k_{G_2}$ as right $kG_2$-modules.
This means by the definition of $S_2$ that the bimodule $\mathcal M$ transposes $k_{G_1}$ 
to $k_{G_2}$. Thus, the adjunction in \cite[Line 9 on p.105]{Rou01} implies that
$\Hom_{(k(G_1\times G_2))^{\text{op}}}(\mathcal M, k_{(G_1\times G_2)})\,{\not=}\,\{ 0 \}$.
Therefore $\mathcal M\cong\Sc(G_1\times G_2,\Delta P)=M$ by \cite[Exercise (27.5)]{The95}.
We are done.
\end{proof}

\begin{Remark}\label{Rem4.4}
On the right-hand side of line 2 of \cite[Lemma~3.1(b)]{KL20b},
$\Sc(\overline G\times\overline H, \Delta\overline P)$ must be replaced by
$\Sc(\overline G\times\overline H, \Delta\overline P)\oplus\mathcal N$ for a
{\it possibly non-zero} $k(\overline G\times\overline H)$-module $\mathcal N$
as  in (\ref{Prop3.3(b)}).
As a result, the proof of \cite[Proposition 3.3(b)]{KL20b} as given in \cite{KL20b}  
holds only in the case in which $\mathcal{N}=\{0\}$. 
However, Theorem~\ref{Prop3.3} now proves that 
 \cite[Proposition~3.3(b)]{KL20b} is correct, also in the case in which $\mathcal{N}\neq\{0\}$.
As a consequence, \cite{KL20b} and \cite[proof of Proposition 5.2]{KLS20}, where
\cite[Proposition 3.3(b)]{KL20b} are used, are not affected and remain true with the given proofs.
Moreover, we would like to mention that the results of \cite{KS23}, 
together with further explicit calculations, give an alternative way to establish
the validity of \cite[Proposition~3.3(b)]{KL20b} 
in special cases, e.g. when the defect groups are generalised quaternion or semi-dihedral $2$-groups.
\end{Remark}

\vspace{4mm}
\noindent {\bf Acknowledgements.} 
{\small 
The authors are thankful to the referees for helpful comments which led to a substantial 
clarification of this manuscript. They would like to thank 
Gerhard Hiss and Gunter Malle for useful information on the modular representation 
theory of the finite groups of Lie type involved in this article.
They would  like to thank  Naoko Kunugi
for pointing out a gap in the proof of \cite[Proposition~3.3(b)]{KL20b}
(see Appendix~\ref{App:A}), as well as Markus Linckelmann and Yuanyang Zhou
for answering several of their questions on Puig's theory.}


\end{document}